\documentclass[12pt]{amsart}
\usepackage{epsfig}
\usepackage{tikz,tikz-cd}
\usepackage{amssymb, amsmath}
\usepackage{amsfonts,graphics,epsfig}
\usepackage{mathrsfs}
\usepackage{pb-diagram, pb-xy}
\usepackage[all]{xy}
\usepackage{comment}
\usepackage{tikz}
\usepackage{setspace}
\usepackage[pagewise]{lineno}
\usepackage[left=1.5in,right=1.5in,top=1.5in,bottom=1.5in]{geometry}
\usepackage{color}   
\usepackage{hyperref}
\hypersetup{
    colorlinks=true, 
    linktoc=all,     
    linkcolor=red,
        citecolor=blue,  
}
\usepackage{tcolorbox}
\usepackage{enumitem}
\usepackage{soul}

\DeclareMathOperator{\Pic}{\mathrm{Pic}}
\DeclareMathOperator{\Ex}{\mathrm{Ex}}

\def\mcF{\mathcal F}

\def\M{\mathbf M}
\def\N{\mathbf N}
\def\A{\mathbf A}
\def\C{\mathbf C}

\newtheorem{theorem}{Theorem}[section]
\newtheorem{lemma}[theorem]{Lemma}
\newtheorem{proposition}[theorem]{Proposition}
\newtheorem{corollary}[theorem]{Corollary}

\theoremstyle{remark}
\newtheorem{remark}[theorem]{Remark}
\theoremstyle{definition}
\newtheorem{definition}[theorem]{Definition}
\theoremstyle{definition}
\newtheorem{example}[theorem]{Example}
\numberwithin{equation}{section}
\theoremstyle{remark}

\theoremstyle{definition}

\newcommand{\mbR}{\mathbb{R}}
\newcommand{\mbC}{\mathbb{C}}

\newcommand{\mbQ}{\mathbb{Q}}

\def\mbN{\mathbb{N}}

\def\>{\geq}

\def\subset{\subseteq}

\newcommand{\lrd}{\lfloor}
\newcommand{\rrd}{\rfloor}

\newcommand{\D}{\Delta}

\def\mcF{\mathcal{F}}
\def\mcG{\mathcal{G}}

\def\Supp{\operatorname{Supp}}
\def\dim{\operatorname{dim}}

\def\Ex{\operatorname{Ex}}

\def\Pic{\operatorname{Pic}}

\def\hor{\operatorname{\hor}}
\def\ver{\operatorname{\ver}}

\def\ver{\operatorname{\textsubscript{\rm ver}}}
\def\hor{\operatorname{\textsubscript{\rm hor}}}

\author{Priyankur Chaudhuri}
\address{Dipartimento di Matematica F. Enriques, Universita degli Studi di Milano, Via Cesare Saldini `
50, 20133 Milano, Italy.}
\address{Yau Mathematical Sciences Center, Tsinghua University, Beijing 100084, China.}
\email{pkurisibang@gmail.com}
\email{chaudhurip@mail.tsinghua.edu.cn}

\author{Roktim Mascharak}
\address{Physics Building, UCL, Gower St, London WC1E 6BT}
\email{roktim.mascharak.23@ucl.ac.uk}

\title[MMP for corank one foliations]{Log canonical minimal model program for corank one foliations on threefolds}

\begin{document}

\begin{abstract}
 If $(X, \mcF, \D)$ is a projective rank two foliated log canonical triple such that $(X,B)$ is klt for some $0 \leq B \leq \D$, we show that we can run a $(K_\mcF +\Delta)$-MMP and any such MMP terminates with either a minimal model or Mori fiber space. Next, we establish a Bertini type lemma and adjunction for generalized foliated quadruples. Using these, we extend the full log canonical MMP to the setting of rank two NQC generalized foliated quadruples. Finally, we apply the generalized MMP to study the relation between different minimal models, namely, any two minimal models of a given foliated log canonical triple can be connected by a sequence of flops and in the boundary polarized case, the minimal models are good and only finitely many in number. 
\end{abstract}
\maketitle

\section{Introduction}
Let $X$ be a normal projective variety with possibly mild singularities. Then, thanks to the minimal model program, it is well known that we can apply a sequence of birational maps $X_0:=X \dashrightarrow X_1 \dashrightarrow X_2 \dashrightarrow \cdots \dashrightarrow X_n$, called divisorial contractions and flips, and this process is expected to end with a variety $X_n$ such that either 
\begin{itemize}
    \item $K_{X_n}$ is nef (then $X_n$ is called a minimal model of $X$), or
    \item there exists a surjective morphism $\phi: X_n \to Y$ with positive dimensional general fiber such that $-K_{X_n}$ is $\phi$-ample ($\phi: X_n \to Y$ is called a Mori fiber space).
\end{itemize}
This way, in order to understand the birational geometry of $X$, it suffices to study $X_n$.\\ 

In recent years, it has been observed that a similar story unfolds if we replace the canonical divisor of the variety $X$ with that of a foliation $\mcF$ on $X$. Indeed, one can define singularities of the foliation $\mcF$ by studying how its canonical divisor $K_\mcF$ changes under birational modifications of $X$ (the relevant definitions are recalled in the next section). In case $\mcF$ has mild singularities, we expect to be able to construct a finite sequence of divisorial contractions and flips $(X, \mcF) \dashrightarrow (X_1, \mcF_1) \dashrightarrow \cdots \dashrightarrow (X_n, \mcF_n)$ (where $\mcF_n$ denotes the transformed foliation on $X_n$) contracting or flipping only curves contained in leaves of the foliation, such that either $K_{\mcF_n}$ is nef or there exists a surjective morphism $f:X_n \to Z$ with positive dimensional general fiber such that $-K_{\mcF_n}$ is ample over $Z$ and $\mcF_n$ descends to a foliation on $Z$.\\

We present here a brief historical overview of the subject. The MMP for rank one foliations on surfaces was carried out by McQuillan \cite{McQ} and Brunella \cite{Bru}. The program for rank two foliations with dlt singularities (F-dlt) on threefolds was carried out by Cascini, Spicer and Svaldi in \cite{Spi}, \cite{CS} and \cite{SS22}. For rank one foliated threefolds, see \cite{CS20}; we remark that the geometry of the MMP here has some significant differences with the rank two case. Analogously to the classical case, the largest class of foliated singularities for which the MMP is still expected to hold are the log canonical (F-lc) ones. Roughly speaking, a log canonical foliation singularity is dlt if there aren't ``too many leaves" passing through it. While, as one might expect, any foliated lc singularity can be birationally modified to one that is dlt \cite[Theorem 8.1]{CS}, extending the MMP from the dlt to the lc setting is not straightforward. We remark that our techniques differ from those of the foundational works \cite{Spi}, \cite{CS} and \cite{SS22}. This is because dlt foliations have non-dicritical singularities \cite[Theorem 11.3]{CS}; in particular, this allows us to compare foliated discrepancy with the classical one and hence to relate $K_\mcF \cdot R$ with $K_X \cdot R$ when $R\subset \overline{NE}(X)$ is a $K_\mcF$-negative extremal ray (see for example, \cite[Lemma 8.14, 8.15]{Spi} for $R$ of divisorial type and \cite[Section 6]{CS} for $R$ of flipping type). When the foliation has dicritical singularities, such techniques don't work and other ideas are needed. Our approach to the log canonical MMP has actually been inspired by that of \cite{CS3} which deals with the algebraically integrable case. \\

We briefly explain some of our motivations behind extending the MMP to the log canonical setting. Many naturally occurring foliations, for example, foliations on $\mathbb{P}^n$ induced by linear projections are log canonical, but not dlt. More generally, it has been shown by Araujo and Druel \cite[Proposition 5.3]{AD13} that Fano foliations are never dlt. This makes the log canonical MMP important for studying questions related to their birational geomtry. Another limitation of dlt foliations appears in boundedness and moduli theory; see the recent work \cite[Section 3.2]{SSV}. Here if one considers foliations $\mcF$ with $K_\mcF$ big, it becomes necessary to run MMP for canonical divisors of the form $K_\mcF+\epsilon K_X$ for $\epsilon>0$ small but fixed; see \cite[Section 3]{SS23}. The induced foliations appearing on the canonical models of such divisors (see \cite[Corollary 3.4]{SS23}) may not have dlt singularities even if the starting foliation does. On the other end, sometimes it is necessary to run the MMP for ''adjoint foliated structures" of the form $tK_\mcF+(1-t)K_X$ for $t \in [0,1]$; see \cite{M7'} (indeed, even in this paper, we have run disguised versions of this type of MMP). In this case, the MMP need not preserve dltness of the foliation. The aim of this article is to establish the log canonical minimal model program for corank one foliations on threefolds with at worst klt singularities and provide some applications. We expect these results to be useful for the development of moduli theory for corank one foliations on threefolds and studying questions on boundedness of adjoint foliated structures (see \cite[section 9]{M7'}). The following is our first main result.

\begin{theorem}
    Let $(X, \mcF, \Delta)/U$ be a rank two projective lc foliated triple, where $\pi:X \to U$ is a projective morphism such that $\dim X =3$ and $(X,B)$ is klt for some $ 0 \leq B \leq \Delta$. Then we can run a $(K_\mcF+\Delta)$-MMP over $U$. Moreover, any such MMP $(X,\mcF,\Delta) \dashrightarrow (X_1, \mcF_1,\Delta_1) \dashrightarrow \cdots$ terminates with an lc foliated triple $(X_n, \mcF_n,\Delta_n)$ satisfying one of the following:
    \begin{enumerate}
        \item if $K_\mcF+\Delta$ is pseudoeffective over $U$, then $K_{\mcF_n}+\Delta_n$ is nef over $U$.
        \item if $K_\mcF+\Delta$ is not pseudoeffective over $U$, then there exists a contraction $X_n \to Z$ over $U$ with $\dim Z < \dim X_n$ whose fibers are tangent to $\mcF_n$ such that $K_{\mcF_n}+\Delta_n$ is anti-ample over $Z$.
    \end{enumerate}
\end{theorem}

In particular, we have shown that \cite[Question 8.4]{SS22} on the existence of log canonical flips for foliations has a positive answer in our setting. Using the theory of toric foliations \cite{CC23}, we also construct an explicit example of such a flip; see Example \ref{exflip}. The foliation in our example has dicritical singularities along the flipping curve. This contrasts the behaviour of rank one foliations which can not have any log canonical centers along an extremal curve of divisorial or flipping type; see \cite[Corollary 8.4]{CS20}.\\

We also extend this result, establishing the full MMP for (NQC) generalized foliated pairs (called generalized foliated quadruples in this  work). It has become apparent in recent years that generalized foliated quadruples (gfqs) are indispensible for the log canonical MMP; see for example, \cite{MMPlcintegrable}. Section $8$ is devoted to the MMP for (NQC) lc gfqs. The existence of this MMP relies on an elementary perturbation trick, which in many cases, can be used as a substitute for the failure of Bertini's Theorem. We state it here since this result may be of independent interest.

\begin{lemma}
    Let $(X, \mcF, \D, \M)$ be a dlt rank two generalized foliated quadruple, where $X$ is a $\mbQ$-factorial klt threefold. Let $A$ be an ample $\mbR$-divisor on $X$. Then there exists $\Theta \geq 0$ such that $(X, \mcF, \Theta)$ is a foliated lc triple and $K_\mcF+\Theta \sim_\mbR K_\mcF+\D+\M_X+A$.
\end{lemma}

The termination of the MMP for rank two gfqs, however, poses several technical challenges, mainly because of the failure of adjunction on invariant centers and the infinitude of log canonical centers for foliations. The main result of this section is the full MMP for lc gfqs:
\begin{theorem}
    Given a corank one lc gfq $(X, \mcF, \D, \M)/U$ where $(X, B, \M)$ is gklt for some $0\leq B \leq \D$ and $\dim X=3$, we can run a $(K_\mcF+\D+\M_X)$-MMP over $U$ and any such MMP terminates with a minimal model or a Mori fiber space.
\end{theorem}

In the case of dlt gfqs, the proof proceeds by showing that any sequence of flips $\phi_i:(X_i, \mcF_i, \D_i, \M)\dashrightarrow (X_{i+1}, \mcF_{i+1}, \D_{i+1}, \M)$ is eventually disjoint from all lc centers of the gfq (Special Termination). We first establish an adjunction type result for generalized foliated quadruples (see Proposition \ref{adj}) and use it to set up an inductive approach (based on the dimension of the lc centers) to Special Termination. \\

Here also, the behaviour of rank two gfqs contrasts that of rank one. In case $(X,\mcF, \D, \M)$ is a $\mbQ$-factorial rank one gfq, any $(K_\mcF+\D+\M_X)$-negative extremal curve $C \subset X$ satisfies $(\M_X \cdot C) \geq 0$ as has been proved recently in \cite[Proposition 4.2]{mengchu}. This relies on an extension lemma for vector fields, proved by Bogomolov and McQuillan; see \cite[Lemma 4.1]{mengchu}. Such results are not applicable to corank one foliations. Thus extending the MMP from foliated triples to gfqs seems to be much more challenging in the corank one case. Note that the minimal model program for generalized pairs in dimension $3$ has been completed only recently; see for example \cite{Chen-Tsak}.\\

We mention one more potential application of the MMP for lc gfqs developed here, that to studying connectedness of non klt loci of foliated triples $(X, \mcF, \D)$ when $-(K_\mcF+\D)$ is nef. See \cite[Theorem 3.2]{SS22} for some previous work in this direction and \cite[Theorem 3.2]{Bircon}, \cite[Theorem 1.1]{FS} for the classical case. \\

Note that some of our terminology differs from earlier literature on birational geometry of foliations: we call objects of the form $(X, \mcF, \Delta)$ foliated triples as in \cite{Liu},\cite{MMPlcintegrable} (rather than foliated pairs as in \cite{CS}, \cite{Spi}) and their generalized pair counterparts as generalized foliated quadruples as in \cite{Liu}, \cite{MMPlcintegrable} (rather than foliated generalized pairs as in \cite{Chaudas} and earlier versions of this paper).\\

Now we present some applications of the generalized log canonical MMP for foliations. The first important application is the following basepoint free theorem (see Theorem \ref{bpft} for a more general version):
\begin{theorem}\label{bpfintro}
    Let $(X,\mcF, \Delta)/U$ be a rank two projective lc foliated triple such that $(X,B)$ is klt for some $ 0 \leq B \leq \Delta$, $\dim X=3$ and $\pi: X \to U$ is a projective morphism. Let $L$ be a nef over $U$ $\mbR$-divisor on $X$ such that $L-(K_\mcF+\Delta)$ is ample over $U$. Then $L$ is semiample over $U$. 
\end{theorem}

A similar result for algebraically integrable foliations has recently been obtained in \cite{MMPlcintegrable}. It thus seems reasonable to expect that on varieties admitting a foliated log Fano  structure, any nef divisor is semiample. We remark that based on recent results obtained by Cascini et al. \cite[Theorem 2.4.2]{M7'}, it seems reasonable to expect that such varieties are in fact of log Fano type. Since the proof of \cite[Theorem 2.4.2]{M7'} involves running MMP for adjoint foliated structures of the form $tK_\mcF+(1-t)K_X$ (or rather, generalized pair versions of these objects), the techniques and results of this paper would turn out to be crucial for establishing such a result.
Combining Theorem \ref{bpfintro} with convex geometric arguments involving Shokurov polytopes (as in \cite{BCHM}), we show that the number of minimal models of a boundary polarized lc foliated triple (i.e. of the form $(X,\mcF, \Delta=A+B)$, where $A$ is ample and $B \geq 0$) is finite; see Theorem \ref{LGLCM}.\\

As another application of the  generalized log canonical MMP, we show that any two foliated minimal models obtained as outcomes of a log canonical MMP can be connected by  a sequence of flops. Note that this was not known before the appearance of this article, even in the dlt case (see \cite{VJ} for some previous work in this direction).

\begin{theorem}\label{flopintro}
Let $(X,\mcF, \Delta)/U$ be a rank two projective $\mbQ$-factorial foliated lc triple such that $\dim X =3$ and $(X, B)$ klt for some $0 \leq B \leq \Delta$ and $\alpha_i: (X, \mcF, \Delta) \dashrightarrow (X_i, \mcF_i, \Delta_i)$, $i=1,2$ two minimal models obtained as outcomes of some $(K_\mcF +\Delta)$-MMPs over $U$, say $\alpha_i:X \dashrightarrow X_i$. Then the induced birational map $\alpha: X_1 \dashrightarrow X_2$ can be realized as a sequence of $(K_{\mcF_1}+\Delta_1)$-flops over $U$.
    
\end{theorem}

For proving Theorem \ref{bpfintro} and Theorem \ref{flopintro}, we have to deal with divisors of the form $K_\mcF+ B+A$, where $A$ is ample. Since a lc foliated triple can have infinitely many log canonical centers, it is perhaps not surprising that the analogues of classical Bertini-type results can fail: if $(X,\mcF, B)$ is lc and $A$ an ample $\mbR$-divisor, it may not be possible to find $\Delta \geq 0$ such that $K_\mcF+\Delta \sim_\mbR K_\mcF+B+A$ and $(X,\mcF, \Delta)$ is lc. As observed earlier in \cite{Chaudas}, the category of generalized foliated quadruples (called generalized foliated pairs in loc. cit.) is flexible enough to deal with such pathologies: instead of considering $(X,\mcF, B+A)$ as an usual pair, we can think of $(X,\mcF,B,A)$ as a generalized foliated quadruple (gfq in short), where we put $A$ in the moduli part; note this gfq is automatically log canonical; see Definition \ref{def:g-pair}. Indeed, it has recently become apparent that gfqs are the right category for carrying out the log canonical MMP for foliations; see for instance \cite{MMPlcintegrable}. Because of this necessity, we need to develop the minimal model program for gfqs. Theorem \ref{bpfintro} implies that minimal models of a boundary polarized foliated lc triple are always good. However, it is worth pointing out that foliated minimal models need not be good in general; see for example \cite[Example 5.4, 5.5]{ACSS}. It would be interesting to find weaker conditions on the boundary which enforce the goodness of minimal models. At present, a minimal model theory for arbitrary foliations on higher dimensional varieties seems out of reach, mainly owing to the absence of a log resolution theorem for foliations.

\section{Preliminaries}
\begin{definition} [Basics on Foliations; see \cite{Spi}, \cite{Dru}]
Let $X$ be a normal quasiprojective variety. A \emph{foliation} $\mcF$ on $X$ is a coherent subsheaf $\mcF \subset T_X$ of the tangent sheaf which is closed under Lie brackets and such that $T_X/\mcF$ is torsion free. Given a foliation $\mcF$ on $X$, $\rm{rank}(\mcF)$ is by definition its rank as a coherent sheaf and $\rm{corank(\mcF):=\dim X-\rm{rank}(\mcF)}$. When $X$ is smooth, the \emph{singular locus} of $\mcF$ is the locus where $\mcF$ fails to be a sub-bundle of $T_X$ (When $X$ is not smooth see \cite[Defination 3.4]{AD}). It is a big open subset of $X$ whose complement has codimension at least $2$. In particular, there exists a big open $U \subset X$ where $X$ and $\mcF$ are both smooth. The \emph{canonical divisor} of $\mcF$, denoted $K_\mcF$ is then the Zariski closure of det$(\mcF|_U)^*$. A subvariety $W \subset X$ is called \emph{$\mcF$-invariant} if for any local section $\partial$ of $\mcF$ over some $U \subset X$ open, $\partial(I_{W\cap U}) \subset I_{W \cap U}$, where $I_{W\cap U}$ is the defining ideal. If $P \subset X $ is a prime divisor, then we define $\epsilon(P)=0$ if $P$ is $\mcF$-invariant and $\epsilon(P)=1$ otherwise.\\

Let $f:X \dashrightarrow Y$ be a dominant rational map between normal varieties and $\mcF$ a foliation on $Y$. Then as in \cite[Section 3.2]{Dru}, we can define the pullback foliation $f^{-1}\mcF$. The pullback of the zero foliation on $Y$ is known as the foliation induced by $f$. Such foliations are called \emph{algebraically integrable}. 
If $f$ is a morphism, note that the foliation induced by $f$ is nothing but the relative tangent sheaf $T_{X/Y}$ over the smooth locus of $f$. If $f: X \dashrightarrow Y$ is birational and $\mcF$ a foliation on $X$, then we have an induced foliation on $Y$ defined by $g^{-1}\mcF$ where $g:= f^{-1}$.\\

If $f:X \to Y$ is an equidimensional morphism of normal varieties and $\mcF$ is a foliation on $X$ such that $\mcF =f^{-1}\mcG$ for some foliation $\mcG$ on $Y$, then $K_\mcF \sim_\mbQ f^*(K_{\mcG})+K_{X/Y}-R$, where $R := \sum_{P: \epsilon(P)=0}(f^*P-f^{-1}P)$ (here $P$ ranges over all $\mcG$-invariant prime divisors in $Y$ and $f^{-1}P:= (f^*P)_{red}$). See \cite[2.9]{Dru17} for details.\\

Let $\mcF$ be a corank one foliation on $X$. We say that a subvariety $V \subset X$ is \emph{tangent} to $\mcF$ if for any birational morphism $\pi:X' \to X$ and any divisor $E$ on $X'$ with center $V$, $E$ is $\mcF'$-invariant, where $\mcF'$ is the pullback of $\mcF$.
\end{definition}

We now explain how to extend the classical definitions of singularities of generalized pairs \cite[Section 4]{BZ} to the foliated setting. They were first introduced by \cite{Liu}. For the reader unfamiliar with the relevant terminology on b-divisors, we refer to \cite{BZ}.

\begin{definition}[Generalized foliated quadruples and their singularities]\label{def:g-pair} A \emph{generalized foliated quadruple} $(X,\mcF, B,\M)/U$ consists of the data of a normal projective variety $X$ equipped with a projective contraction morphism $\pi:X \to U$ to a projective variety $U$, a foliation $\mcF$ on $X$, a $\mathbb{R}$-divisor $B \geq 0$ and a b-nef over $U$ $\mathbb{R}$-divisor $\M$ on X such that $K_\mcF+B+\M_X$ is $\mathbb{R}$-Cartier.
Let $\pi:Y \to X$ be a higher model of $X$ to which $\M$ descends and $\mcF_Y$ the pulled back foliation on $Y$. Define $B_Y$ by the equation \begin{center}
    $K_{\mcF_Y}+B_Y+\M_Y = \pi^*(K_\mcF+B+\M_X)$,
\end{center} 
where $\mcF_Y:=\pi^{-1}\mcF$. If for any prime divisor $E$ on any such $Y$, mult$_EB_Y \leq \epsilon(E)$, then we say that $(X,\mcF,B,\M)$ is \emph{log canonical} (lc in short). For $\pi:X' \to X$ a higher model and prime divisor $E\subset Y$, we define its \emph{discrepancy} $a(E, X, \mcF, B,\M):=-\rm{coeff}_EB_{X'}$, where $B_{X'}$ is defined by $K_{\mcF'}+B_{X'}+\M_{X'}= \pi^*(K_\mcF+B+\M_X)$.
In this paper, we use the abbreviation \emph{gfq} for generalized foliated quadruples and when $\M=0$, we refer to $(X, \mcF, B)$ as a \emph{foliated triple}. A gfq $(X, \mcF, B,\M)$ is called \emph{NQC} if for any higher model $\pi:Y \to X$ of $X$ to which $\M$ descends, $\M_Y$ is a positive linear combination of nef over $U$ $\mbQ$-Cartier $\mbQ$-divisors (in other words, $\M_Y$ is  NQC over $U$).
\end{definition}
For the convenience of the reader, next we include the definition of foliated log smooth pairs, which will be needed to define dlt gfqs.
\begin{definition}\cite[Definition $3.1$]{CS}
Given a corank one foliated triple $(X,\mathcal{F},B)$ we say that $(X,\mathcal{F},B)$ is foliated log smooth provided the following hold:
\begin{enumerate}
    \item $(X,B)$ is log smooth 
    \item $\mathcal{F}$ has simple singularities \cite[Definition 2.8]{CS}, and 
    \item If $S$ is the support of non $\mathcal{F}$-invariant components of $B$, $p\in S$ is a closed point and $\Sigma_1,\Sigma_2...,\Sigma_k$ are the (possibly formal) $\mathcal{F}$-invariant divisors passing through $p$, then $\Supp S \cup \Sigma_1\cup...\cup \Sigma_k$ is a normal crossing divisor
\end{enumerate}
\end{definition}
The existence of a foliated log smooth model has been instrumental in the development of the MMP for corank one foliations on threefolds \cite{Spi}, \cite{CS}. Next, we define dlt singularities for generalized foliated quadruples. They serve as a close approximation of simple singularities, but with the added feature of being preserved by the MMP.
\begin{definition}
    Let $(X,\mathcal{F},B,\M)$ be a generalized foliated quadruple. We say that $(X,\mathcal{F},B,\M)$ is \textit{dlt} if
    \begin{enumerate}
        \item $(X,\mcF, B, \M)$ is lc, and
        \item There exist a resolution $\pi:Y\rightarrow X$ on which $\M$ descends satisfying the following properties:\begin{enumerate}\item if $B_Y$ is defined by $K_{\mcF_Y}+B_Y+\M_Y=\pi^*(K_\mcF+B+\M_X)$, then $(Y,\mathcal{F}_Y,B_Y)$ is foliated log smooth, and
        \item $\pi$ only extracts divisors $E$ with $\text{mult}_EB_Y<\epsilon (E)$. In other words, $\pi$ only extracts klt places of $(X, \mcF, B, \M)$.
        \end{enumerate}
        
    \end{enumerate}
\end{definition}

Finally, we recall the definition of flips and log canonical models for gfqs:

\begin{definition}
    Let $(X,\mathcal{F},\D,\M)/U$ be an lc generalized foliated quadruple. A projective birational morphism $f:X \rightarrow Z$ over $U$ is called a $(K_\mcF+\D+\M_X)$-\emph{flipping contraction} if $\rho(X/Z)=1$, $f$ is small (i.e. has exceptional locus of codim at least $2$) and $-(K_\mcF+\D+\M_X)$ is ample over $Z$. Let $f^{+}: X^{+} \rightarrow Z$ be a projective birational morphism over $U$ from a normal projective variety $X^{+}$ and $\phi: X \dashrightarrow X^{+}$ the induced birational map. Then $f^{+}$ is a \emph{flip} of $f$ if $f^{+}$ is small and $\phi_*(K_\mcF+\D+\M_X)$ is $\mathbb{R}$-Cartier and ample over $Z$. \\
    
    Let $\phi: (X, \mcF,\D, \M) \dashrightarrow (X',\mcF', \D', \M)$ be a birational contraction over $U$ which is $(K_\mcF+\D+\M_X)$-non-positive. Then $(X',\mcF', \D', \M)$ is called the \emph{log canonical model} of $K_\mcF+\D+\M$ over $U$ if $K_{\mcF'}+\D'+\M_{X'}$ is ample over $U$.
\end{definition}

\section{Cone Theorem}
\begin{theorem}\label{lccone}
    Let $(X,\mathcal{F},\Delta)/U$ be a projective lc rank two foliated triple, where $\dim X =3$. Then there exists a countable collection of rational curves $\lbrace C_i\rbrace_{i \in I} $ on $X$ tangent to $\mcF$ such that.
    \begin{enumerate}
        \item $\overline{NE}(X/U)=\overline{NE}(X/U)_{(K_{\mathcal{F}}+\Delta)\geqslant 0}+\sum_{i \in I} \mbR_+ [C_i]$
        \item $-6\leqslant (K_X+\Delta)\cdot C_i<0$
        \item For any relatively ample divisor $H$ over $U$, $(K_{\mathcal{F}}+\Delta+H)\cdot C_i\leqslant 0$ for all but finitely many $i$.
    \end{enumerate}
\end{theorem}
\begin{proof}
   We first prove the theorem in the case $U$ is a point. Thanks to \cite{Spi}, we have the full cone theorem for dlt rank two foliated triples. To prove this in the log canonical setting, first we will need a linear algebra lemma; see \cite[Lemma 3.1]{wald}.
    \begin{lemma}\label{lin}
        Let $f:V\rightarrow W$ be a surjective linear map of finite dimensional vector spaces. Suppose $C_V\subset V$ and $C_W \subset W$ are closed convex cones of maximal dimension and $H\subset W$ is a linear subspace of codimension 1. Assume that $f(C_V)=C_W$ and $C_W\cap H \subset \partial C_W$. Then $f^{-1}H\cap C_V=f^{-1} (H\cap C_W)\cap C_V$ and $f^{-1}H\cap C_V \subset \partial C_V$.
    \end{lemma}
    
    $(X,\mathcal{F},\Delta)$ is log canonical, hence by \cite[Theorem $8.1$]{CS} there exists a F-dlt modification $(Y,\mathcal{G},\Delta_Y)$ such that $f^*(K_{\mathcal{F}}+\Delta)=K_{\mathcal{G}}+\Delta_Y$, where $f:Y\rightarrow X$ is the induced morphism. There is a surjective map of vector spaces $f_*:N_1(Y)\rightarrow N_1(X)$, which induces a surjection $f_*(\overline{NE}(Y))=\overline{NE}(X)$. By the cone theorem for the $\mbQ$-factorial foliated dlt triple $(Y,\mcG, \Delta_Y)$, we know that there is a countable collection of rational curves $C_i^Y$ satisfying 
    \[
    \overline{NE}(Y)=\overline{NE}(Y)_{(K_{\mathcal{G}}+\Delta_Y)\geqslant 0}+\sum \mbR^+\cdot [C_i^Y].
    \]
Let $C_i$ be the countable collection of rational curves on $X$ given by $f_*C_i^Y$ with reduced structure. Suppose $\overline{NE}(X)\neq \overline{NE}(X)_{(K_{\mathcal{F}}+\Delta)\geqslant 0}+\sum \mbR_{\geqslant 0}[C_i]$. Then there is some $\mbR$-Cartier divisor $D$ which is positive on the right hand side of the above equation and non-positive somewhere on $\overline{NE}(X)$. Let $A$ be an ample divisor and $\lambda=\text{inf}\lbrace t:D+tA \text{ is nef }\rbrace$. Then by Kleiman's criterion $D+\lambda A$ takes value $0$ somewhere on $\overline{NE}(X)\setminus \lbrace0\rbrace$. By replacing $D$ by $D+\lambda A$ we may assume that $D_{=0}$ intersects $\overline{NE}(X)$ non-trivially. So $D_{=0}$ cuts out some extremal face $F$ of $\overline{NE}(X)$. By Lemma \ref{lin} we have 
\[
F_Y:=f_*^{-1}F\cap \overline{NE}(Y)=f_*^{-1}D_{=0}\cap \overline{NE}(Y)
\]
is some non-empty extremal face of $\overline{NE}(Y)$, which is $(K_{\mathcal{G}}+\Delta_{Y})$-negative away from the lower dimensional face $f_*^{-1}(0)$. But any such extremal face contains a negative extremal ray $R=\mbR^+\cdot [C_i^Y]$. Then $D_{=0}$ contains one of the $C_i$ which contradicts the assumption of inequality.\\

Now to show the inequality in part $(2)$ observe that
\[
0<\frac{(K_{\mathcal{F}}+\Delta)\cdot C_i}{(K_{\mathcal{F}}+\Delta)\cdot f_*C_i^Y}\leqslant 1
\] 
as $C_i$ is $f_*C_i^Y$ with reduced structure. Now we know that $(K_{\mathcal{F}}+\Delta)\cdot C_i$ and $(K_{\mathcal{F}}+\Delta)\cdot f_*C_i^Y$ are both negative. Hence by projection formula and the $\mbQ$-factorial dlt cone theorem, this gives the desired bound: $(K_\mcF + \Delta) \cdot C_i \geq (K_\mcG+ \Delta_Y) \cdot C_i^{Y} \geq -6$ .\\

To prove part $(3)$ we first show that the negative extremal rays do not accumulate in $(K_{\mathcal{F}}+\Delta)<0$. Suppose otherwise, some sequence $R_i$ converging to a $(K_{\mathcal{F}}+\Delta)$-negative ray $R$. Let $R_i^Y$ be an extremal ray in $\overline{NE}(Y)$ satisfying $f_*R_i^Y=R_i$. Such a ray exists by definition of $R_i$. By compactness of the unit ball in $\overline{NE}(Y)$ a subsequence of $R_i^Y$ converges to a ray $R^Y$. This must satisfy $f_*R^Y=R$, and so by projection formula, $R^Y$ is $(K_{\mathcal{G}}+\Delta_Y)$-negative. This contradicts the cone theorem for dlt foliated triples.\\

Finally, let $A$ be an ample $\mbR$-divisor on $X$. Suppose there are infinitely many $C_i$ with $(K_{\mathcal{F}}+B+A)\cdot C_i<0$. By compactness, some subsequence of the corresponding $R_i$ converges to a ray $R$. This must satisfy $(K_{\mathcal{F}}+\Delta+A)\cdot R\leqslant 0$, but $R\subset \overline{NE}(X)$ so this implies $(K_{\mathcal{F}}+\Delta)\cdot R<0$, which contradicts that the negative extremal rays do not accumulate in $(K_{\mathcal{F}}+\Delta)<0$.\\

Now, we prove the theorem in the relative set-up. We can deduce this from the projective case using standard techniques as in \cite[Section $3.6$]{KM98}. Namely, given any $\zeta\in \overline{NE}(X/U)$ we can write $\zeta=\eta+\sum r_j[C_j]$, where a priori $\eta\in \overline{NE}(X)_{K_{\mathcal{F}}+\Delta\geqslant 0 }$ and $C_j$ are $(K_{\mathcal{F}}+\Delta)$-negative. Then we can argue as in \cite[3.28]{KM98} to get $\eta\in \overline{NE}(X/U)$ and $\pi_*[C_j]=0$. This gives us the relative cone theorem.

\end{proof}
\section{Contraction Theorem}
\begin{theorem}\label{CT}
    Let $(X,\mcF, \Delta)/U$ be a corank one log canonical foliated triple such that $(X,B)$ is klt for some $0 \leq B \leq \Delta$ and $\dim X =3$. Let $R \subset \overline{NE}(X/U)$ be a $(K_\mcF+\Delta)$-negative exposed extremal ray (see \cite[Definition 6.5]{Spi}). Then there exists a contraction $c_R: X \to Z$ over $U$ associated with $R$, where $Z$ is a normal projective variety over $U$ of klt type. Moreover, $c_R$ satisfies the following properties:
    \begin{enumerate}
        \item If $L\in \Pic X$ is such that $L \equiv_Y 0 $, then there exists $M\in \Pic Z$ with $c_R^*M=L$,
        \item If $X$ is $\mbQ$-factorial and $c_R$ is a divisorial or Fano contraction, then $Z$ is also $\mbQ$-factorial.
    \end{enumerate}
    \end{theorem}
    \begin{proof}
    For simplicity, we deal with the absolute case first (i.e $U$ is a point). 
        Many of the ideas of the proof were inspired by \cite[Theorem 3.2]{CS3}. Since $(X,B)$ is klt, we have a small $\mbQ$-factorialization $h:\Tilde{X}\rightarrow X$. Let $\pi:\overline{X}\rightarrow X$ be a F-dlt modification which exists due to \cite[Theorem $8.1$]{CS}. By construction of the F-dlt modification, we may assume that $\pi$ factors through $h$ and we denote $\pi':\overline{X}\rightarrow \Tilde{X}$ the induced morphism. Let $\overline{\mathcal{F}}$ be the induced foliation on $\overline{X}$. Let $\Gamma=\pi_*^{-1}B$ and $\overline{\Delta}=\pi_*^{-1}\Delta +\sum \epsilon(E_i)E_i$, where the last sum runs over all $\pi$-exceptional divisors. Then we have \[
        K_{\overline{\mathcal{F}}}+\overline{\Delta}=\pi^*(K_{\mathcal{F}}+\Delta)
        \]
        We may also write 
        \[
        K_{\overline{X}}+\Gamma+E_0=\pi^*(K_X+B)+F_0
        \]
        
    where $E_0$ and $F_0$ are effective $\pi$-exceptional divisors with no common components. Note that $(\overline{X}, \Gamma +E_0)$ is klt by \cite [Lemma 3.16]{CS} (using the fact that $(\overline{\mcF}, \Gamma + (E_0)_{non-inv})$ is F-dlt). Since $\Tilde{X}$ is $\mbQ$-factorial there exists a $\pi'$-exceptional Cartier divisor $B_0\geqslant 0$ such that $-B_0$ is $\pi'$-ample. Since $(\overline{X}, \Gamma +E_0)$ is klt there exists $\delta>0$ sufficiently small such that $(\overline{X},\Gamma+E_0+\delta B_0)$ is klt. Let $E:=E_0+\delta B_0$ and $F:=F_0+\delta B_0$. In particular we have that \begin{center}
        
    $K_{\overline{X}}+\Gamma+E=\pi^*(K_X+B)+F$ ($*$)
    \end{center}
    There exists a nef $\mbR$-Cartier divisor $H_R$ such that $H_R^{\perp}\cap \overline{NE}(X)=R$ and if $D$ is a Cartier divisor on $X$ such that $R\subset D^{\perp}$ then $(H_R+tD)^{\perp}\cap \overline{NE}(X)=R$ for any sufficiently small $t>0$ (thanks to Theorem \ref{lccone}). Since $R$ is exposed, we may write $H_R=K_{\mathcal{F}}+\Delta+A$ where $A$ is an ample $\mbR$-divisor. Let $\overline{A}=\pi^*A$. By construction we know that $K_{\overline{\mathcal{F}}}+\overline{\Delta}+\overline{A}=\pi^*H_R$ is nef. Now consider any curve $C'$ such that $[\pi(C')]\in R$ (note that the curve exists by Cone theorem for $(\mcF, \Delta)$). Then we have that $(K_{\overline{\mathcal{F}}}+\overline{\Delta}+\lambda\overline{A})\cdot C'<0$ for any $\lambda<1$, hence $K_{\overline{\mathcal{F}}}+\overline{\Delta}+\lambda\overline{A}$ is not nef.\\
    
    First assume that $H_R$ is not big. Let $\nu$ be its numerical dimension. We define $D_i:=\pi^*H_R$ for all $1\leqslant i\leqslant \nu+1$ and $D_i := \overline{A}$ for $\nu+1 < i \leq 3$. Then we have $D_1\cdot D_2\cdot D_3=(\pi^*H_R)^{\nu+1}\cdot\overline{A}^{3-\nu-1}=0$ and \[
    -(K_{\overline{\mathcal{F}}}+\overline{\Delta})\cdot D_2\cdot D_3>0
    \]
Now, by \cite[Corollary $2.28$]{Spi}, through a general point of $\overline{X}$, there exists a rational curve $\overline{C}$ which is tangent to $\overline{\mathcal{F}}$ such that $\pi^*H_R\cdot \overline{C}=0$. Thus it follows that $(K_{\overline{\mathcal{F}}}+\overline{\Delta}+\overline{A})\cdot \overline{C}=\pi^*H_R\cdot \overline{C}=0$. Hence $\mbR_{\geqslant 0}[\pi(\overline{C})]=R$, with $loc(R)=X$. By the same arguments as in the proof of \cite[Lemma 8.12]{Spi}, there exists a fibration $\phi: \overline{X} \to Z$ with $\dim Z < 3$ such that $\overline{\mcF}$ descends to a foliation $\mathcal{G}$ on $Z$ and the curves $\overline{C}$ as above are contained in the fibers of $\phi$. With this at hand, there exists a big open subset $U \subset Z$ which is smooth such that $(K_{\mathcal{\overline{F}}}+\overline{\Delta})|_{f^{-1}(U)}\sim_{\mbQ,U} (K_{\overline{X}}+\overline{\Delta}+G)|_{f^{-1}(U)}$, where $G= \phi^{-1}(B_\phi)_{\mcG\rm{-inv}}$ is the inverse image of the $\mathcal{G}$-invariant part of the branch divisor of $\phi$. In particular, $(K_\mathcal{\overline{F}}+\overline{\Delta}) \cdot \overline{C} = (K_{\overline{X}}+\overline{\Delta}+G)\cdot \overline{C} <0$, as $\overline{C}$ is a covering family of curves. Let $G'$ denote the reduced sum of $\overline{\mcF}$-invariant $\pi$-exceptional divisors which are not components of $G$. Then if $\epsilon$ is sufficiently small, $(K_{\overline{X}}+\overline{\Delta}+G +\epsilon G') \cdot \overline{C} <0$. We can adjust $\delta$ above such that $\Gamma +E \leq \overline{\Delta}+G+\epsilon G'$. With this, we have $(K_{\overline{X}}+\Gamma+E) \cdot \overline{C} \leq(K_{\overline{X}}+\overline{\Delta}+G +\epsilon G')\cdot \overline{C} <0$. This along with equation ($*$) above implies that $(K_X+B) \cdot R <0$. So we can use the klt contraction theorem to contract $R$. \\


We now assume that $H_R$ is big. Since $(\overline{X},\overline{\mathcal{F}},\overline{\Delta})$ is a dlt foliated triple,
we can run a $(K_{\overline{\mathcal{F}}}+\overline{\Delta})$-MMP with the scaling of $\overline{A}=\pi^*A$.  This MMP produces a sequence of $(K_{\overline{\mathcal{F}}}+\overline{\Delta})$-flips and divisorial contractions $\phi_i:\overline{X}_i\dashrightarrow \overline{X}_{i+1}$, and a sequence of rational numbers $\lambda_1=1\geqslant\lambda_2 \cdots$, such that if $\overline{\mathcal{F}_i}$ is the induced foliation on $\overline{X}_i$ and $\overline{A}_i$ and $\overline{\Delta}_i$ are strict transforms of $\overline{A}$ and $\overline{\Delta}$ on $\overline{X}_i$ then $(\overline{X}_i,\overline{\mathcal{F}_i},\overline{\Delta}_i)$ is dlt foliated triple, where $\lambda_i:=\rm{inf}\lbrace t\geqslant 0|K_{\overline{\mathcal{F}}}+\overline{\Delta}_i+t\overline{A}_i$ is nef$\rbrace$. As any sequence of $(K_{\overline{\mathcal{F}}}+\overline{\Delta})$-MMP terminates by \cite{SS22}, there exists an $i_0$ such that $\lambda_{i_0}<1$. After possibly replacing $i_0$ by a smaller number, we may assume $i_0$ is the smallest such positive integer. Set $\lambda:=\lambda_{i_0}$, $\overline{X}':=\overline{X}_{i_0}$ and $\phi:\overline{X}\dashrightarrow \overline{X}'$ be the induced birational map. Let $\overline{\mcF}':=\phi_*\overline{\mcF}$, $\Gamma':=\phi_*\Gamma$ and similiarly $E'$, $\overline{\Delta}'$ and $\overline{A}'$. By \cite[Lemma 3.33, Lemma 3.16]{CS}
it follows that $(\overline{X}',\Gamma'+E')$ is klt.\\

By definition of the MMP  with scaling and by our choice of $i_0$, we have that $K_{\overline{\mcF}'}+\overline{\Delta}'+t\overline{A}'$ is nef for all $\lambda\leqslant t\leqslant 1$ and each step of this MMP up until $\overline{X}'$ is $(K_{\overline{\mathcal{F}}}+\overline{\Delta}+t\overline{A})$-negative for $t<1$ and $(K_{\overline{\mathcal{F}}}+\overline{\Delta}+\overline{A})$-trivial. Thus $\phi_*\pi^*H_R=K_{\overline{\mathcal{F}}}+\overline{\Delta}+\overline{A}$ is nef. By \cite[Lemma $3.1$]{CS3}, we have a containment $(K_{\overline{\mcF}'}+\overline{\Delta}'+t\overline{A}')^{\perp}\cap \overline{NE}(\overline{X}')\subset (\phi_*\pi^*H_R)^{\perp}\cap \overline{NE}(\overline{X}')$ for all $\lambda<t\leqslant 1$.\\
Fix a rational number $\lambda'$ such that $\lambda<\lambda'<1$ and a sufficiently small rational number $s>0$ such that
\begin{enumerate}
    \item $K_{\mathcal{F}}+\Delta+\lambda' A$ is big
    \item if we set $A_0=(1-\lambda')A-s(K_X+\Delta)$, then $A_0$ is ample, the stable base locus of $H_R-A_0$ coincides with $\mathbb{B}_+(H_R)$ and $H_R-A_0$ is positive on every extremal ray of $\overline{NE}(X)$ except $R$.
    \item $s<\frac{1}{2m\dim X}$ where $m$ is the Cartier index of $K_{\overline{\mathcal{F'}}}+\overline{\Delta}'+\lambda'\overline{A}'$, and
    \item if we set $K=(K_{\overline{\mathcal{F}}}+\overline{\Delta}+\lambda'\overline{A})+s(K_{\overline{X}}+\Gamma+E)$, then $K$ is big and $\phi$ is $K$-negative.
\end{enumerate}
Set $K'=\phi_*K$. Since $K_{\overline{\mcF}'}+\overline{\Delta}'+\lambda'\overline{A}'$ is nef and big by our choice of $\lambda'$, we may run a $K'$-MMP which is $(K_{\overline{X}'}+\Gamma'+E')$-negative and $(K_{\overline{\mathcal{F}}}+\overline{\Delta}'+\lambda'\overline{A}')$-trivial. Call this MMP $\psi:\overline{X}'\dashrightarrow \overline{X}''$ and let $\overline{\mcF}'':=\psi_*\overline{\mcF}'$, $\Gamma''=\psi_*\Gamma'$, $\overline{\Delta}''=\psi_*\overline{\Delta}'$, $E''=\psi_*E'$, $\overline{A}''=\psi_*\overline{A}'$ and $K''=\psi_*K'$. Note that this MMP is $\overline{H_R'}$-trivial by \cite[Lemma 3.1]{CS3}.\\

We have that $(\overline{X}'',\Gamma''+E'')$ is klt and $\frac{1}{s}K''-(K_{\overline{X}''}+\Gamma''+E'')$ is big and nef. Thus by classical basepoint free theorem we have $K''$ is semi ample.\\

By our choice of $\lambda'$ and $s$, the $\mbQ$-divisor $A_0$ is ample and the stable base locus of $(H_R-A_0)$ is same as $\mathbb{B}_+(H_R)$. We can write $K=\pi^*(H_R-A_0)+sF$ and it is easy to check that the restricted base locus of $K$ is exactly the union of Supp $F$ with the preimage of augmented baselocus of $H_R$, which we denote by $Z$. Thus by \cite[Lemma $2.1$]{CS3} we have that divisorial part of Supp $F\cup Z$ is contracted by $\psi\circ\phi$, and since $\Ex \pi'=$Supp $F$ it follows that the induced map $f:X\dashrightarrow \overline{X}''$ is an isomorphism in codimension one. In particular we have that 
$\Gamma''=f_*B$. Note that since $\Tilde{X}$ and $\overline{X}''$ are both $\mbQ$-factorial, it follows that $\rho(\tilde{X})=\rho(\overline{X}'')$. By \cite[Lemma 3.1]{CS3}, we have $\psi\circ\phi:\overline{X}\dashrightarrow\overline{X}''$ is $\pi^*H_R$-trivial.\\

Next, we show the existence of contraction $c_R:X\rightarrow Z$ associated to $R$. If $(K_X+B)\cdot R<0$ then we can contract $R$ by the classical basepoint free theorem. Suppose not, then we have $(K_X+B)\cdot R\geqslant 0$. If we have $(K_X+B)\cdot R=0$ then $H_R+\epsilon(K_X+B)$ is nef and big for $\epsilon>0$ small enough (with null locus $R$), hence semiample by the basepoint free theorem. Hence $R$ can be contracted in this case. Thus we may assume $(K_X+\Delta)\cdot R>0$, hence there exists $c>0$ such that $c(K_X+\Delta)\cdot \zeta=A_0\cdot \zeta$ for all curves $\zeta$ with $[\zeta]\in R$. By our choice of $c$ we have that if $m$ is a sufficiently large positive integer and we set $\tilde{H_R}=(c(K_X+\Delta)-A_0)+mH_R$, then $\tilde{H_R}$ is nef and $\tilde{H_R}^{\perp}\cap \overline{NE}(X)=R$ (here we have used the fact that $(H_R-A_0)|_{\overline{NE}(X) \setminus R}>0$). In particular, we have that $\pi^*H_R^{\perp}\cap \overline{NE}(\overline{X})=\pi^*\tilde{H}_R^{\perp}\cap\overline{NE}(\overline{X})$. We know that $\psi\circ\phi$ is $(\pi^*H_R)$-trivial, so it is also $(\pi^*\tilde{H}_R)$-tirvial. In particular $f_*\tilde{H_R}$ is nef. As $f_*(K_X+B)=K_{\overline{X}''}+\Gamma''$ we have that $f_*\tilde{H_R}-c(K_{\overline{X}''}+\Gamma '')=f_*(\tilde{H_R}-A_0)+(m-1)f_*\tilde{H_R}$ is big and nef, since $f_*(\tilde{H_R}-A_0)=(\psi\circ\phi)_*K$ is nef and $f_*\tilde{H_R}$ is big and nef. Since $(\overline{X}'', B'')$ is klt and $\mbQ$-factorial, we may therefore apply the basepoint free theorem (c.f. \cite[Theorem $3.9.1$]{BCHM}) to conclude that $f_*\tilde{H}_R$ is semi ample. It follows that $\tilde{H}_R$ is itself semi ample, so we can take $c_R:X \to Z$ as its semiample fibration.\\

Our next goal is to show descent of relatively numerically trivial line bundles and the preservation of $\mbQ$-factoriality under divisorial and Fano contractions. Let $c_R: X \to  Z$ be the contraction obtained above. If $c_R$ has positive dimensional general fiber, then as observed above, $(K_X+B) \cdot R <0$, so everything follows from the classical case; see for example \cite[Lemma 3.18]{KM98}. So we assume $c_R$ is birational. First, we show descent. Let $L \in \Pic X$ such that $L \equiv _{Z}0$. The descent of $L$ and $\mbQ$-factoriality of $Z$ follows from the classical case when $(K_X+B) \cdot R <0$. If $(K_X+B)\cdot R =0$, then $L-(K_X+B)$  is nef and big over $Z$ (since $c_R$ is birational in this case), therefore $L$ descends to $Z$ by the basepoint free theorem. From now on we assume that $(K_X+B) \cdot R>0$. In the notation of the above paragraph, note then that $\psi \circ \phi$ is $L$-trivial (since it is $\overline{H_R}$-trivial and $H_R^\perp \subset L^\perp$). Thus it is enough to show that $L'':= f_*L$ descends to $Z$. First, we note that $\phi$ is a full $(K_{\overline{\mcF}}+\overline{\D}+\lambda \overline{A})$-MMP. Thus $\phi$ contracts $B_-(K_{\overline{\mcF}}+\overline{A}+\lambda \overline{A})$. In particular if loc $R$ is a divisor, then $\phi$ contracts its strict transform. As observed above, $\psi$ contracts $\Ex \pi$, thus $B_+(\overline{H_R}'')$ is an union of curves and the induced morphism $\pi'': \overline{X}'' \to Z$ contracts them. In particular, $\pi''$ is small. By construction of $c_R$,  there exists $\Theta' \geq 0$ such that $(Z, \Theta')$ is klt. Let $\Theta''$ be defined by $K_{\overline{X}''}+\Theta''= \pi''^*(K_{Z}+\Theta')$.\\

We claim that there exists $G'' \geq 0$ such that $-G''$ is $\pi''$-ample. Indeed, fix ample divisors $A''$ on $\overline{X}''$ and $A'$ on $Z$. Let $U \subset \overline{X}''$ big open such that $\pi''|_U$ is an isomorphism. We can choose $s>0$ such that $(s(\pi''^*A')-A'')|_U$ is ample. Choose $G''_U    \in |(s(\pi''^*A')-A'')|_U|_\mbQ$ and let $G''$ be its Zariski closure. Since $G'' \sim_\mbQ s(\pi''^*A')-A''$, $-G''$ is $\pi''$-ample. This proves the claim.\\

If $\delta>0$ is sufficiently small, then $(\overline{X}'', \Theta''+\delta G'')$ is klt. Now $L''-(K_{\overline{X}''}+\Theta''+\delta G'')$ is ample over $Z$. Thus, thanks to  the classical basepoint free theorem \cite[Theorem 3.9.1]{BCHM}, $L''$ is semiample over $Z$. Hence, there exists $M \in \Pic Z$ such that $L''= \pi''^*M$. In particular, $L$ descends to $Z$.\\

We now show that if $X$ is $\mbQ$-factorial and $c_R$ a divisorial contraction, then $Z$ is $\mbQ$-factorial. The arguments are standard. Let $E\subset X$ be the $c_R$-exceptional divisor and let $D'$ be a Weil divisor on $X'$ whose strict transform on $X$ we denote by $D$. Since $\rho(X/X')=1$, there exists $\alpha \in \mbQ$ such that $D-\alpha E \equiv_{X'}0$. Thus, by the above observation, there exists a $\mbQ$-Cartier $\mbQ$-divisor $M$ on $X'$ which pulls back to $D-\alpha E$. Pushing forward to $Z$  gives $M=D'$, thus $D'$ is $\mbQ$-Cartier.\\

Now to prove the theorem for projective morphism $\pi:X\rightarrow U$ of projective varieties, if we follow the proof of the absolute case, we just need existence and termination of relevant relative MMPs and the classical base-point free theorem in this set-up. The existence and termination of classical MMP and classical base-point free theorem for a projective morphism from a normal projective three-fold is well known. Thanks to \cite[Theorem $2.1$]{SS22} we have existence and termination of the foliated MMP for dlt foliated triples, as a corollary (c.f. \cite[Theorem $2.4$]{SS22}) we also have F-dlt modification in this relative set-up. Since $\pi:X\rightarrow U$ is a projective morphism, similarly as in the case of $U$ being a point, we can use the arguments concerning the sign of the intersection number $(K_X+B)\cdot R$ for a negative extremal ray $R\in \overline{NE}(X/U)$. Hence the proof of the contraction theorem in the absolute case readily works for the relative case.

\end{proof}



\section{Existence of flips}
In this section, we prove the existence of flips for a rank two lc foliated triple $(X,\mathcal{F},\Delta)$ where $\dim X=3$. In fact we have the following more general result on the existence of relative log canonical models. At the end of this section, we present an example of a foliated log canonical flip. The results of this section answer \cite[Question 8.4]{SS22} in our setting. 
\begin{theorem}\label{lcflip}
    Let $(X, \mcF, \Delta)/U$ be a projective foliated lc triple of rank two such that $(X, B)$ is klt for some $B \leq \Delta$, and $\pi:X\rightarrow U$ be a projective morphism of projective varieties. Let $f: X \to Z $ be the contraction of a $(K_\mcF+\Delta )$-negative extremal ray over $U$. Then the log canonical model $(X^+, \mcF^+, \Delta^+)$ of $(K_\mcF+\Delta)$ over $Z$ exists. Moreover, letting $B^+$ denote the strict transform of $B$ on $X^+$, $(X^+, B^+)$ is klt and $(X^+,\mcF^+, \Delta^+)$ is lc.
\end{theorem}
  \begin{proof}

  Let $g:\overline{X} \to X$ be a F-dlt modification of $(X,\mcF, \Delta)$ (which exists due to \cite[Theorem $2.4$]{SS22}), $\Gamma:= g_*^{-1}B$ and write \begin{center}
     $ K_{\overline{X}}+ \Gamma +E_0 = g^*(K_X+B) +F_0$,
  \end{center}
  where $E_0, F_0 \geq 0$ are $g$-exceptional divisors without common components. We may assume $g$ factors through a small $\mbQ$-factorialization $ \overline{X} \xrightarrow{g'} X'\to X$. Let $G \geq 0$ be an exceptional divisor on $\overline{X}$ which is anti-ample over $X'$. Using \cite[Lemma 3.16]{CS}, there exists $ \delta>0$ such that $(\overline{X}, \Gamma+E_0+\delta G)$  is klt. Let $E:=E_0+\delta G$, $F :=F_0+\delta G$, so we have $K_{\overline{X}}+\Gamma+E = g^*(K_X+B)+F$. Let $ \overline{\Delta}:= \pi_*^{-1}\Delta+ \sum_i \epsilon(E_i)E_i$, where the last sum is over all $g$-exceptional divisors. Let $\phi: \overline{X} \dashrightarrow \overline{X}'$ be a $(K_{\overline{\mcF}}+\overline{\Delta})$-MMP over $Z$. We can choose $s>0$ such that if $K:= K_{\overline{\mcF}}+\overline{\Delta}+s(K_{\overline{X}}+\Gamma+E)$, then $\phi$ is a $K$-MMP. Letting $K':= \phi_*K$, let $\psi: \overline{X}'\dashrightarrow \overline{X}''$ be a $K'$-MMP over $Z$ and note that we can replace $s$ by something possibly smaller such that $\psi$ is $(K_{\overline{\mcF}'}+\overline{\Delta}')$-trivial. Then $K'':= \psi_*K'$ is semiample over $Z$ by the classical basepoint free theorem. We note that $\psi \circ \phi$ contracts the relative diminished base locus $B_{-}(K/Z)$ by \cite[Lemma 2.1]{CS3}. We now proceed to show that $\psi \circ \phi$ contracts all $g$-exceptional divisors. Let $H$ be an ample over $Z$ divisor on $X'$. We can pick $ \epsilon >0$ such that $g'^*H - \epsilon G$ is ample over $Z$, $B_-(K/Z)=B((K+ g'^*H- \epsilon G)/Z)$ and such that $ \epsilon < s \delta$. Note that
  \begin{center}
      $K+ g'^*H - \epsilon G = g'^*[(K_{\mcF'} +\Delta' +H)+ s(K_{X'}+B')]+sF- \epsilon G$,
  \end{center}
   where $B', \Delta'$ denote the strict transforms of $B, \Delta$  on $X'$, $ sF-\epsilon G \geq 0$ by our choice of $\epsilon$ and that the exceptional divisor $\Ex g'\subset Supp (sF-\epsilon G)$. Thus $|K+g'^*H -\epsilon G|_\mbR=g^*|(K_{\mcF'} +\Delta' +H)+ s(K_{X'}+B')|_\mbR +sF- \epsilon G$ and in particular, $ \Ex g' \subset B_-(K/Z)$. We have thus shown that $ \psi \circ \phi$ contracts $\Ex g'$.\\


   Now we construct the relative $(K_{\mcF}+ \Delta)$-log canonical model $X^+$. Suppose $f$ is the contraction of a $(K_\mcF+\Delta)$-negative extremal ray $R$. For this, we may assume $(K_X+B) \cdot R \geq 0$ (if $(K_X+B) \cdot R <0$, the existence of the desired model is given by \cite[Theorem 1.1]{Bir12}). Then there exists some $\alpha \geq 0, \beta >0$ and an ample over $Z$ divisor $A$ such that $K_X+B \sim _{\mbR, Z} \alpha A$ and $-(K_\mcF+\Delta) \sim _{\mbR,Z} \beta A$. Then letting $\overline{A}:= g^*A$ we have $K\sim_{\mbR, Z}(s \alpha- \beta) \overline{A}+ sF$ and we may choose $s$ such that $s \alpha < \beta$. Letting $\overline{X}''$ denote the output of the $K$-MMP over $Z$ as above, note that $F''=0$ (since the MMP contracts all exceptional divisors of $\overline{X} \to X'$). Thus $K''\sim_{\mbR,Z} (s \alpha - \beta)\overline{A}''$ is semiample over $Z$, forcing $K_{\overline{\mcF}''}+\overline{\Delta}'' \sim_{\mbR, Z}-\beta \overline{A}''$ to be semiample over $Z$; in particular, the desired log canonical model is given by its relative semiample fibration $\overline{X}'' \to X^+$. Note that $(X^+, B^+)$ is still klt as the semiample fibration is $(K_{\overline{X}''}+\overline{B}'')$-trivial.

  \end{proof} 
  \begin{remark}
      In the above proof, note that $K_{\mcF'}+\overline{\Delta}'$, being crepant to $K_{\mcF''}+\overline{\Delta}''$, is semiample over $Z$. We will need this observation when we show termination of flips for foliated lc triples.
  \end{remark}
  \begin{remark}
      Note that our proof of the existence of log canonical flips works verbatim in arbitrary dimensions once we have the existence of F-dlt modification and termination of the relative F-dlt MMP over the base of the flipping contraction. 
  \end{remark}
  
  Next we construct an example of a strictly log canonical foliated flipping contraction on a threefold. Our example shows that unlike the case of rank one foliations on threefolds (see \cite[Corollary 8.4]{CS20}), for rank two, the flipping curve can pass through the locus of  dicritical singularities of the foliation. 
  \begin{example}\label{exflip}
   Let $N$ be a lattice of rank $3$ and $\sigma$ the cone generated by vectors $v_1,v_2,v_3,v_4$ generating $N$ as a lattice such that $v_1+v_3=v_2+v_4$. Consider $\Delta_1$ to be the subdivison of $\sigma$ which we get by joining $v_1$ and $v_3$, and $\Delta_2$ to be the subdivison of $\sigma$ we get by joining $v_2$ and $v_4$. Let $X$ be the toric variety corresponding to $\sigma$ and $X(\Delta_i)$ be the toric varieties corresponding to $\Delta_i$ for $i=1,2$. The morphism $X(\Delta_1)\rightarrow X$ induced by the subdivision is the Atiyah flopping contraction with the flopping curve corresponding to the two dimensional face generated by $v_1$ and $v_3$. Note that $X(\Delta_2)\rightarrow X$ is the corresponding flop and the flopped curve in $X(\Delta_2)$ corresponds to the two dimensional face generated by $v_2$ and $v_4$.\\
   
   We will try to realise this flop as a foliated flip. Firstly giving a toric foliation on a toric variety is equivalent to giving a vector subspace $W\subset N\otimes \mbC$. Let $W$ be the complex vectorspace generated by $v_2$ and $v_4$. Consider the corresponding corank one foliation $\mathcal{F}_{W}$ on $X(\Delta_1)$. Let us denote the face generated by $v_1$ and $v_3$ by $\tau$, and the corresponding curve to be $C_{\tau}$. By \cite[Proposition$1.9$]{CC23} we know that $K_{\mathcal{F}_W}=-D_{v_2}-D_{v_4}$, $D_{v_i}$ are divisors on $X(\Delta_1)$ corresponding to the rays generated by $v_i$ for $i=2,4$. Now by \cite[Lemma 6.4.4]{CLS11} we can see that $-(D_{v_2}+D_{v_4})\cdot C_{\tau}<0$. Note that this curve generates an extremal ray in $\overline{NE}(X(\Delta_1))$. Hence $\mbR^+[C_{\tau}]$ is a $K_{\mathcal{F}_{W}}$-negative extremal ray of flipping type. By Cone theorem, $\mathcal{C_{\tau}}$ is $\mathcal{F}_W$ invariant and hence $X(\Delta_1)\rightarrow X$ is a $K_{\mathcal{F}_W}$ flipping contraction. By \cite[Proposition $3.8$]{CC23} $\mathcal{F}_W$ has at worst foliated log canonical singularities as $X(\Delta_1)$ is $\mbQ$-factorial. Now observe that $\tau$ is not contained in $W$ but $W\cap relint(\tau) \cap N$ is non-empty. By \cite[Theorem 1.19]{CC23} $\mathcal{F}_{W}$ is dicritical hence it can not have F-dlt or canonical singularities (see \cite[Theorem $11.3$]{CS}). So $\mathcal{F}_W$ is a strictly log canonical foliation on $X(\Delta_1)$ and $X(\Delta_1)\rightarrow X$ is our desired flipping contraction. Again using \cite[Proposition 6.3.4]{CLS11}, it is easy to see that $X(\Delta_2)\rightarrow X$ is the foliated flip of $X(\Delta_1)\rightarrow X$. Note that the strict transform of $\mathcal{F}_W$ on $X(\Delta_2)$ is non-dicritical by \cite[Theorem 1.19]{CC23}. In particular, the foliation $\mcF_W$ has dicritical singularities only along the flipping curve $C_\tau$. Now, we show that $C_{\tau}$ is a strictly lc center for $\mcF_1$. Let $\D_3$ be the subdivision of $\D_1$ we get by joining $v_2$ and $v_4$. Then we have a morphism $p:X(\D_3)\rightarrow X(\D_1)$ with the exceptional divisor $E$ surjecting onto $C_{\tau}$. By \cite[Proposition $5.10$]{CC23} $E$ is not invariant by $\mcF_3$. Hence $C_{\tau}$ is a strictly lc center for $\mcF_1$. Consider the divisor $\Sigma_1=D_{v_2}+D_{v_4}$ in $X(\D_1)$. Note that $(X(\D_1),\mcF_1,\Sigma_1)$ is a lc foliated triple, and $K_{\mcF_1}+\Sigma_1\sim 0$. Write $p^*(K_{\mcF_1}+\Sigma_1)=K_{\mcF_3}+\Sigma_3+bE$, where $\Sigma_3$ is the strict transform of $\Sigma_1$, which does not contain the support of $E$. We have that $K_{\mcF_3}+\Sigma_3+bE\sim 0$, which implies $\Sigma_3+bE$ is reduced sum of $\mcF_3$ non-invariant torus boundary. From the above information we get that $b=1$, in particular $a(E,\mcF_1,\Sigma_1)=-1$, which implies $(X(\D_3),\mcF_3,\Sigma_3+E)$ is the F-dlt modification of $(X(\D_1),\mcF_1,\Sigma_1)$. As the transformed foliation $\mcF_2$ on $X(\D_2)$ has non-dicritical singularities, the image of $E$ in $X(\D_2)$, which is the flipped curve $C_{\tau'}$, is not tangent to the foliation $\mcF_3$.
  \end{example} 

  \section{Running the MMP}
  Let $(X, \mcF, \Delta)/U$ be a corank one lc foliated triple where $\dim X=3$ and there exists $\Delta \geq B \geq 0$ such that $(X,B)$ is klt. Let $R$ be a $(K_\mcF+\Delta)$-negative extremal ray over $U$. Then by Theorem \ref{CT}, the associated contraction $c_R: X \to Z$ exists. However, in case $X$ is not $\mbQ$-factorial, it is not clear that the strict transforms of $K_\mcF +\Delta$ and $K_X+B$ remain $\mbR$-Cartier on $Z$. This issue can be addressed as follows:\\

  Let $(X', \mcF', \Delta')$ be the log canonical model of $(X, \mcF, \Delta)$ over $Z$; its existence is guaranteed by Theorem \ref{lcflip}. Replace $(X, \mcF, \Delta)$ with $(X', \mcF', \Delta')$ and note that by Theorem \ref{lcflip}, $(X',B')$ is klt. Then by Theorem \ref{CT} and \ref{lcflip}, we can continue running the $(K_{\mcF'}+\Delta')$-MMP. It follows from the proof of Theorem \ref{lcflip} that if $X$ is $\mbQ$-factorial, then so is $X'$.

\section{Termination}
In this section, we prove that any MMP for foliated lc triples eventually terminates.
\begin{theorem}\label{lcterm}
     Starting from a rank two foliated projective lc triple $(X,\mathcal{F},\Delta)/U$ where $\dim X =3$ and $(X,B)$ klt for some $0\leqslant B\leqslant \Delta$, there is no infinite sequence of $(K_\mcF+\D)$-MMP over $U$.
\end{theorem}
\begin{proof}
    Let \begin{equation*}
\xymatrix{    
               (X_0, \mathcal{F}_0,\Delta_0)\ar[dr]_{f_0} \ar@{-->}[rr] &     & (X_1,\mathcal{F}_1,\Delta_1)\ar[dl]^{f_0^+}\ar[dr]_{f_1} \ar@{-->}[rr]& & (X_2,\mathcal{F}_2,\Delta_2)\ar[dl]^{f_1^+} \cdots&\\
                                             &Z_0&                  & Z_1&
                                             }
 \end{equation*} 
 be a sequence of foliated log canonical MMP starting from $(X_0,\mcF_0,\Delta_0):=(X, \mcF, \Delta)$. Now as in the proof of Theorem \ref{lcflip}, we go to a F-dlt modification $(\overline{X_0}, \overline{\mcF_0},\overline{\Delta_0})$ of $(X_0,\mcF_0,\Delta_0)$ and run a $(K_{\overline{\mcF_0}}+\overline{\Delta_0})$-MMP $\phi_0:\overline{X_0} \dashrightarrow \overline{X_0}'$ followed by a $K:=K_{\overline{\mathcal{F}_0}'}+\overline{\Delta_0}'+s(K_{\overline{X}'}+\overline{\Gamma}'+E)$-MMP $\psi_0: \overline{X_0}' \dashrightarrow \overline{X_0}''$ over $Z_0$. For small enough $s$, the whole map $\overline{X}_0\dashrightarrow \overline{X}_0''$ is a $K$-MMP over $Z_0$. Note that $\psi_0$ is a $(K_{\overline{\mathcal{F}_0}'}+\overline{\Delta_0}')$-trivial, partial $(K_{\overline{X}'}+\Gamma_0'+E_0')$-MMP over $Z_0$. Hence if $p:W\rightarrow \overline{X_0}'$ and $q:W\rightarrow \overline{X_0}''$ resolve the locus of indeterminacy of $\psi_0$ we get that $p^*(K_{\overline{\mathcal{F}_0}'}+\overline{\Delta_0}')=q^*(K_{\overline{\mathcal{F}_0}''}+\overline{\Delta_0}'')$. As $K_{\overline{\mathcal{F}_0}''}+\overline{\Delta_0}''$ is semi-ample over $Z_0$, so is $K_{\overline{\mathcal{F}_0}'}+\overline{\Delta_0}'$. $(X_1, \mathcal{F}_1,\Delta_1)$ being the ample model, we have a morphism $\pi_1:\overline{X}_0'\rightarrow X_1$, given by the semiample fibration of $K_{\overline{\mathcal{F}_0}'}+\overline{\Delta_0}'$ over $Z_0$.  Since $(\overline{X_0}', \overline{\mcF_0}', \overline{\Delta_0}')$ is dlt foliated triple, we see that $\pi_1$ is a F-dlt modification of $(X_1, \mcF_1, \Delta_1)$. Now we can again repeat the process starting from $(X_1, \mcF_1, \Delta_1)$ over $Z_1$. Hence we get a sequence of foliated flips and divisorial contractions over $U$, in particular over Spec $\mbC$ starting from a $\mbQ$-factorial dlt triple $(\overline{X}_0,\overline{\mathcal{F}_0},\overline{\Delta_0})$ which must terminate by \cite[Theorem 2.1]{SS22}. Hence the starting sequence of log canonical flips cannot be infinite.
\end{proof}
\section{Minimal model program for generalized foliated quaduples}
In this section, we develop the MMP for NQC generalized foliated quadruples. Owing to failure of Bertini's theorem for foliated lc triples, this general setting seems necessary for establishing foliated versions of classical MMP results: flop connection between minimal models (Theorem \ref{flop}) and the basepoint free theorem \ref{bpft}. First, we record a Bertini-type result which allows us to reduce the contraction theorem and existence of flips for dlt gfqs to the case of lc triples. This result could be of independent interest. Note that all foliated quadruples we consider in this paper will be assumed to be NQC.

\begin{lemma}\label{dlt-lc} Let $(X,\mcF, \Delta, \M)/U$ be a $\mbQ$-factorial dlt rank two gfq, where $X$ is a $\mbQ$-factorial projective threefold and $\pi:X\rightarrow U$ is a projective morphism from $X$. Further, assume there exists $0 \leq B \leq \Delta$ such that $(X,B, \M)$ is klt. Let $A$ be a $\pi$-ample $\mbR$-divisor on $X$. Then there exists $\Theta \geq 0$ such that $(X,\mcF,\Theta)$ is lc and $K_\mcF+\Theta\sim_{\mbR,U} K_\mcF +\Delta+\M_X+A$.

\end{lemma}
\begin{proof}
 First, observe that replacing $\D+A$ with $(1-\epsilon)\D+(A+\epsilon\D)$ where $\epsilon>0$ small, we may reduce to the case $\lfloor \D \rfloor=0$. Let $g: X' \to X$ be a foliated log resolution on which $\M$ descends such that $g$ only extracts foliated klt places of $(X,\mcF, \Delta, \M)$. Write $K_{\mcF'}+\Delta'+\M_{X'} = g^*(K_\mcF +\Delta+\M_X)$ and note that coeff$_E(\Delta')< \epsilon(E)$ for any prime divisor $E$ which is a component of $\Delta'$. We can choose $0 <\delta$ such that $g^*A-\delta E$ is $\pi$-ample for some $E \geq 0$ which is $g$-exceptional and for all $\delta$ sufficiently small.  Letting $A':=g^*A$, we can write $K_{\mcF'}+\Delta'+\M_{X'}+A' = K_{\mcF'}+(\Delta'+\delta E)+\M_{X'}+(A'-\delta E)$ and note that coeff$_{E_i}(\Delta'+\delta E)< \epsilon(E_i)$ for any exceptional divisor $E_i$ possibly decreasing $\delta$ further. Consequently, $(X', \mcF', \Delta'+\delta E)$ is sub dlt with $\lrd{\Delta'+\delta E}\rrd \leq 0$. Note that $\M_{X'}+A'-\delta E$ is ample over $U$, and as $U$ is a projective variety, we can find a sufficiently ample divisor on $U$ and add its pullback to make $\M_{X'}+A'-\delta E$ globally ample on $X'$. Then by \cite[Proposition 3.9]{CS}, for a general member $H' \in |\M_{X'}+A'-\delta E|_{\mbR}$, $H'$ does not contain any lc centers of $(X', \mcF', \Delta'+\delta E, \M_{X'}+A'-\delta E)$. Then letting $\Theta':=\Delta'+\delta E +H'$, $\Gamma' := B'+\delta E+H'$, $(X', \mcF', \Theta')$ is sub-lc. Letting $\Theta:=g_*\Theta'$ finishes the proof.

\end{proof}


\subsection{Minimal model program for dlt gfqs}
Let $(X,\mcF, B, \M)/U$ be a dlt gfq, where $X$ is a $\mbQ$-factorial normal projective threefold and $(X, \overline{B},\M)$ is gklt for some divisor $0 \leqslant\overline{B}\leqslant B$. Let $R \subset \overline{NE}(X/U)$ be a $(K_{\mcF}+B+\M_X)$-negative extremal ray. We can then find an ample $\mbR$-divisor $A$ such that $(K_\mcF+B+\M_X+A) \cdot R <0$. By Lemma \ref{dlt-lc}, $K_\mcF+B+\M_X+A\sim_{\mbR,U} K_\mcF+\Delta$ for some $\Delta\geq 0$ such that $(X,\mcF, \Delta)$ is lc and $X$ is klt. Thus by the contraction theorem and the existence of flips for lc foliated triples, we have a contraction morphism $\phi$ for $R$ and if $\phi$ is of flipping type, the corresponding flip exists. Thus we can always run a $(K_\mcF +B+\M_X)$-MMP over $U$ in the above setting. We proceed to show the termination of this MMP. The following adjunction type lemma turns out to be the main technical ingredient. It will be used for setting up an inductive approach to (special) termination. In what follows, if $D$ is an effective $\mbR$-divisor on a variety $X$, we use the notation $(D)_{\leq 1}$ to denote $\lfloor D \rfloor_{\rm{red}}+\{D\}$.

\begin{proposition}\label{adj}
  Let $(X, \mcF, \Delta, \M)/U$ be a lc rank two gfq where $X$ is a normal projective threefold of klt type and $\pi: X \to U$ is a projective morphism.
  \begin{enumerate}
      \item $T$ be a prime divisor on $X$ with mult$_T\Delta=\epsilon(T)=1$ and normalization $\nu: S \to X$. Then there exists $\Theta_S \geq 0$ on $S$ such that letting $\mcF_S$ denote the restricted foliation on $S$, there exists a lc gfq $(\mcF_S, \Theta_S, \N)$ with $\nu^*(K_\mcF+\Delta+\M_X)= K_{\mcF_S}+\Theta_S+\N_S$.
      
      \item Assume that $(X,\mathcal{F},\D,\M)$ is a $\mbQ$-factorial dlt gfq. Let $D\subset X$ be a prime divisor invariant by $\mcF$, with normalisation $\nu:D^{\nu} \to D$. Then there exists $\Theta_{D^{\nu}} \geq 0$ such that $(K_{\mcF}+\D+\M_X)|_{D^{\nu}}=K_{D^{\nu}}+\Theta_{D^{\nu}}+\N_{D^\nu}$ and $(D^\nu, (\Theta_{D^\nu})_{\leq 1}, \N)$ is generalized lc. 
      
      Moreover, the non generalized log canonical locus $\rm{Supp} (\Theta_{D^\nu})_{>1}$ is contained in the pre-image of intersection of $Sing\mcF$ and $D$ in $D^{\nu}$ and the support of $\nu((\Theta_{D^\nu})_{>1})$ consists of lc centers of $(X, \mcF, \D,\M)$.
      
      \item Assume that $(X,\mathcal{F},\D,\M)$ is a $\mbQ$-factorial dlt gfq. Let $C$ be a one dimensional generalized lc center tangent to $\mathcal{F}$ with normalisation $C^{\nu}$. Then there exists an effective divisor $\Theta_{C^{\nu}}$ such that $(K_{\mcF}+\D+\M_X)|_{C^{\nu}}=K_{C^{\nu}}+\Theta_{C^{\nu}}+\N$. Moreover if $P$ is contained in the support of $\lfloor\Theta_{C^{\nu}}\rfloor$ then $\nu(P)$ is a generalized lc center $(X,\mathcal{F},\D,\M)$, and the coefficients of $\{\Theta_{C^\nu}\}$ belong to a DCC set independent of $X$ and $\mcF$.
  \end{enumerate}
  \end{proposition}
\begin{proof}
        First, replacing $X$ by a small $\mbQ$-factorialization, we may assume $X$ to be $\mbQ$-factorial; in particular, $\M_X$ is $\mbR$-Cartier. Let $g:X'\to X$ be a foliated log resolution of $(X,\mcF, \Delta,\M)$ such that $\M$ descends to a nef over $U$ divisor on $X'$ and $S' := g_*^{-1}S$ is smooth. Let $\N_{S'}:=\M_{X'}|_{S'}$ and $\N_S:=g_*\N_{S'}$. Let $\D'$ be defined by $K_{\mcF'}+\Delta'+\M_{X'}= g^*(K_\mcF+\Delta+\M_X)$ and $\Theta_{S'}$ be defined by $(K_{\mcF'}+\Delta')|_{S'}=K_{\mcF'_{S'}}+\Theta_{S'}$, where $\mcF'_{S'}$ denotes the restricted foliation on $S'$. Since $K_{\mcF'_{S'}}+\Theta_{S'}+\N_{S'}=g^*(K_{\mcF_S}+\Theta_S+\N_S)$, it is enough to show that $(S',\mcF'_{S'}, \Theta_{S'})$ is sub-lc. For this, first note that $(K_{\mcF'}+S')|_{S'}=K_{\mcF'_{S'}}+B_{S'}$, where $B_{S'}\geq 0$ is contained in the locus where $S'$ is tangent to $\mcF'$ \cite[Corollary 3.3]{Spi}. By the definition of foliated log smooth, it follows that $B_{S'}=0$. Then we get $K_{\mcF'_{S'}}+(\D'-S')|_{S'}=g^*(K_{\mcF_S}+\Theta_{S})$. By log smoothness, if $K_{\mcF'_{S'}}+(\D'-S')|_{S'}$ is not sub-lc, then there exists a component $E$ of $\D'-S'$ with $\rm{coeff}_E(\D'-S')=$$a_E$ such that $\epsilon_\mcF(E)=1$, but $\epsilon_{\mcF'_{S'}}(E|_{S'})=0$ and $a_E>0$. We show this can't happen. Let $\mu: Y \to X'$ be the blow up of $X'$ along $E \cap S'$ with exceptional divisor $F$. Then $\epsilon_\mcF(F)=0$. We claim that the foliated discrepancy $a(F,\mcF', \D') \leq -a_E$. By log smoothness, this is true for the usual discrepancy of $F$ with respect to $(X', \D')$. The claim then follows from \cite[Lemma 3.1]{Spi}. 
        Set $\Theta_S= \pi_*\Theta_{S'}$. Then it is easy to check that $\Theta_S \geq 0$. Indeed, it follows by another application of the negativity lemma that the generalized foliated different is at least the usual foliated different (see \cite[Remark 4.8]{BZ} for the arguments) and the latter is effective by \cite[Propostion 3.4]{Spi}. Thus $(S,\mcF_S, \Theta_S, \N)$ is lc.\\
        
        Now, to prove part $2$, we consider a foliated log resolution $h:(Y, \mcF')\rightarrow (X, \mcF)$ of $(X,\mathcal{F},\D+D,\M)$ such that $\M$ descends to $Y$. Letting $D':= h_*^{-1}D$, since $h|_{D'}$ factors through $D^\nu$, we replace $D$ by $D^{\nu}$. Let $\M_Y|_{D'}=:\N_{D'}$ and $\N_{D^{\nu}}:=h_*\N_{D'}$. We define $\D'$ by $K_{\mcF'}+\Delta'+\M_{X'}= h^*(K_\mcF+\Delta+\M_X)$, $\Theta_{D'}$ by $(K_{\mcF'}+\D')|_{D'}=K_{D'}+\Theta_{D'}$ and $\Theta_{D^{\nu}}=h_*\Theta_{D'}$. Then $h^*(K_{D^{\nu}}+\Theta_{D^{\nu}}+\N_{D^{\nu}})=K_{D'}+\Theta_{D'}+\N_{D'}$. Note that $(D', \D'|_{D'})$ is a log smooth lc sub pair. Let $B_{D'}$  be defined by $K_{\mcF'}|_{D'}=K_{D'}+B_{D'}$, then by the proof of \cite[Lemma 8.9]{Spi}, it follows that $B_{D'}$ is supported on the codimension two components of $\rm{Sing} \mcF'$ contained in $D'$. Since $\Theta_{D'}=B_{D'}+\D'|_{D'}$, it follows that the non-lc locus of the sub pair $(D',\Theta_{D'})$ is supported on codimension two components of the singular locus of $\mcF'$ contained in $D'$. Notice that $\mathcal{F}'$ has simple singularities, which forces the non-lc centers of $(D',\Theta_{D'})$ to be log canonical centers of the foliation $\mathcal{F}'$ (see for example, the proof of \cite[Lemma 3.3]{CS}). Let $C'$ be such a component of $\rm{Sing}\mcF'$. We claim that $\pi(C')$ is a one-dimensional component of $\rm{Sing}\mcF$. Suppose not, then the foliation $\mcF$ is smooth along $\pi(C')$, in particular at an analytic neighbourhood of $\pi(C')$ the foliation admits a holomorphic first integral. That implies $\mcF'$ admits a holomorphic first integral at an analytic neighbourhood of $C'$. As $C'$ is in the non-lc locus of $(D',\Theta_{D'})$, the foliation has simple singularity of second kind along $C'$, i.e. a saddle node singularity \cite[Definiton 2.8]{CS}. Hence, $\mcF'$ does not admit any holomorphic first integral along $C'$, which is a contradiction, and we have proved our claim. Now observe that $\mcF$ can't have terminal singularities along $\pi(C')$; otherwise by \cite[Corollary 5.15]{SS22}, it would have a holomorphic first integral, which would lead to a contradiction as above. 
       As the image of the non lc locus of $(D',\Theta_{D'})$ is the non glc locus of $(D^{\nu},\Theta_{D^{\nu}},\N_{D^{\nu}})$, it is supported on the pre image of some one dimensional components of $\rm{Sing}\mcF$ contained in $D$. Now we observe that $\pi(C')$ is tangent to $\mcF$. Indeed, this is clear if the generic point of $C'$ is not contained in $\rm{Ex} (\pi)$. Otherwise, $C'$ is contained in some $\pi$-exceptional divisor(s) which is forced to be invariant by \cite[Remark 3.2]{CS}. But then by \cite[Remark 2.16]{CS}, all exceptional divisors with center $\pi(C')$ are $\mcF'$-invariant, thus proving the tangency of $\pi(C')$. Since $\mcF$ can't have terminal singularities along $\pi(C')$, it follows that $\pi(C')$ is a lcc of $\mcF$, in particular, a lcc of the gfq $(X,\mcF,\D,\M)$.     
       This proves that the non glc locus of $(D^{\nu},\Theta_{D^{\nu}},\N_{D^{\nu}})$ is supported on the pre-image of one-dimensional lc centers of the dlt gfq $(X,\mcF,\D,\M)$ in $D^{\nu}$.\\
       
       Let $\hat{X}$ denote the formal completion of $X$ along $D$ and $T= \sum_{i=1}^k T_i$ denote the (possibly formal) divisors on $X$ intersecting $D$. Note that $T$ is analytically $\mbQ$-Cartier as in the proof of \cite[Lemma 3.18]{CS}.
        Consider a log resolution $g:(\hat{Y},\Delta_{\hat{Y}}+T_{\hat{Y}}+D_{\hat{Y}})\rightarrow(\hat{X},\Delta+D+T)$ of $(\hat{X},\D+D+T)$, such that induced foliation $\mcF_{\hat{Y}}$ on $\hat{Y}$ has simple singularities and the b-divisor $\M$ descends on $Y$, where $\D_{\hat{Y}}, T_{\hat{Y}}, D_{\hat{Y}}$ denote the strict transforms of $\D, T, D$ respectively. We can write $K_{\hat{Y}}+\D_{\hat{Y}}+D_{\hat{Y}}+T_{\hat{Y}}+ \sum E_{i0}+\sum E_{j1}+\M_{\hat{Y}}=g^*(K_{\hat{X}}+\D+T+D+\M_{\hat{X}})+\sum b_{i0}E_{i0}+\sum b_{j1}E_{j1}$, $E_{i0}$ and $E_{j1}$'s are $\mathcal{F}_Y$-invariant and non-invariant divisors respectively. Similarly we can write the same equation in terms of $K_{\hat{\mcF}}$ where $\hat{\mcF}$ is the restriction of $\mcF$ to $\hat{X}$: $K_{\mcF_{\hat{Y}}}+\D_{\hat{Y}}+\M_{\hat{Y}}+\sum E_{1j}=g^*(K_{\hat{\mcF}}+\D+\M_{\hat{X}})+\sum a_{1j}E_{1j} +\sum a_{0i}E_{0i}$, where $a_{1j}, a_{0i} \geq 0$ for all $i,j$. Note that $(X,\mcF,\D)$ has dlt singularities by the negativity lemma; in particular $\mcF$ is non dicritrical by \cite[Theorem 11.3]{CS}. Arguing as in the proof \cite[Lemma $8.14$]{Spi}, it follows that $b_{1j}\geq a_{1j}$ and $b_{0i} \geq a_{0i}$. Restricting the last equation to $D'$ we get that $(K_{\mcF_{\hat{Y}}}+\D_{\hat{Y}}+\sum(1-a_{j1})E_{j1}+\M_{\hat{Y}})|_{D'}=K_{D'}+\Theta_{D'}+\N_{D'}$. Now, consider $(K_{\hat{Y}}+\D_{\hat{Y}}+\sum(1-a_{1j})E_{1j}+D_{\hat{Y}}+T_{\hat{Y}}+\sum E_{i0}+\M_{\hat{Y}})|_{D'}=K_{D'}+(\Theta_{D'})_{\leqslant 1}+\N_{D'}$; note that the different in this case is $(\Theta_{D'})_{\leq 1}$ by \cite[Lemma $8.9$]{Spi}. Our target is to find some boundary $\Theta'_{D^{\nu}}$ such that $(D^{\nu},\Theta'_{D^{\nu}},\N_{D^{\nu}})$ is generalized lc. Let $\Theta'_{D'}$ be defined by $(K_{\hat{Y}}+\D_{\hat{Y}}+\sum(1-b_{1j})E_{1j}+D_{\hat{Y}}+T_{\hat{Y}}+\sum (1-b_{0i}) E_{0i}+\M_{\hat{Y}})|_{D'}=K_{D'}+\Theta'_{D'}+\N_{D'}$, let $\pi_*\Theta'_{D'}=:\Theta'_{D^{\nu}}$ and note that $\pi_*\Theta_{D'}=\Theta_{D^\nu}$. Note that $1-b_{1j}\leq 1-a_{1j}$ and $\pi_*(E_{0i}|_{D'})=0$ for all $E_{0i}$ with $a_{0i}>0$, where the latter follows from the efectivity of $\Theta_{D^{\nu}}$. We claim that $E_{i0}$ with $a_{i0}=0$ are also lc places of $(\hat{X}, \D+D+T,\M)$. For this, observe that $(X, \mcF, \D, \M)$ is log smooth at the generic point of the image of any such $E_{0i}$ since the latter is an lc center. By the proof of \cite[Prop 3.9]{CS}, it follows that any such lc center is a strata of $\rm{Sing} \mcF$ or a strata of $\lfloor \D \rfloor$. By \cite[Appendix]{Cano}, it then follows that any $E_{0i}$ with $a_{0i}=0$ is an lc place of $(\hat{X}, \D +D+T, \M)$.
        It follows that $(\Theta_{D^{\nu}})_{\leq 1}\geq \Theta'_{D^{\nu}}$. By adjunction of varieties $(D',\Theta'_{D'},\N)$ is generalized log canonical. We also have that $K_{D'}+\Theta'_{D'}+\N_{D'}=g^*(K_{\hat{X}}+\D+T+D+\M_{\hat{X}})|_{D'}=g^*(K_{D^{\nu}}+\Theta'_{D^{\nu}}+\N_{D^{\nu}})$, which implies $(D^{\nu},\Theta'_{D^{\nu}},\N)$ is generalized log canonical. Finally, we claim that $(D^{\nu},(\Theta_{D^{\nu}})_{\leq 1},\N)$ is also lc. Suppose not, then the support of $(\Theta_{D^{\nu}})_{\leq 1}-\Theta'_{D^{\nu}}$ consists of some non lc centers of $(D^{\nu},\Theta_{D^{\nu}},\N)$. Let us denote a component of this effective divisor by $Z$. We have already proved in the above paragraph that $Z$ is tangent to $\mcF$. However, the effective divisor $(\Theta_{D^\nu})_{\leq 1}-\Theta'_{D^\nu}$ is supported on $g_*(\sum_j(b_{1j}-a_{1j})E_{1j}|_{D'})$, where each component of this divisor is transverse to the foliation $\mcF$ as $E_{1j}$s are $\mcF'$-non-invariant exceptional divisors. This is a contradiction. Hence, $(D^{\nu},(\Theta_{D^{\nu}})_{\leq 1},\N_{D^{\nu}})$ is also lc. \\
        
        Now we prove the third part. By negativity lemma, it follows that $(X, \mcF, \D)$ is a dlt triple.
        Let $\pi: (X', \mcF', \D', \M)\to (X, \mcF, \D, \M)$ be a foliated log resolution of $(X, \mcF, \D, \M)$ such that $\M$ descends on $X'$ and such that $\pi$ does not extract $C$ (such resolution exists because $(X,\mcF,\D,\M)$ is a dlt gfq and $C$ is a lc center). Here $\D'$ is defined by $K_{\mcF'}+\D'+\M_{X'}=\pi^*(K_{\mcF}+\D+\M_X)$.  Note that $(X,\mcF,\D)$ is log smooth at the generic point of $C$. Then we can find an invariant surface $S$, such that $C$ is contained in $S$. Let $C'$ and $S'$ denote the strict transforms of $C$ and $S$ respectively. If $C\subset \rm{Sing}\mcF$, we can take $S$ to be a strong separatrix along $C$, thus $S'$ is a strong separatrix along $C'$. Replacing $S'$ with its normalization and writing $(K_{\mcF'}+\D'+\M_{X'})|_{S'}= K_{S'}+\Theta_{S'}+\mathbf{Q}_{S'}$, we have $\rm{coeff}_{C'}\Theta_{S'}=1$. Writing $(K_{S'}+\Theta_{S'}+\mathbf{Q}_{S'})|_{C'}=K_{C'}+\Theta_{C'}+\N_{C'}$, we have $\Theta_C=\pi_*\Theta_{C'}$. If $p \in \lfloor\Theta_C \rfloor$, then there exists $p' \in \lfloor \Theta_{C'}\rfloor$ such that $\pi(p')= p$. Then it's enough to show that $p'$ is an lc center of $(X', \mcF', \D',\M)$. For this, note that $K_{C'}+(\Theta_{C'})_{\leq 1}+\N_{C'}= (K_{S'}+(\Theta_{S'})_{\leq 1}+\mathbf{Q}_{S'})|_{C'}$ and $(S',(\Theta_{S'})_{\leq 1},\mathbf{Q})$ is glc. By inversion of adjunction, it follows that $p' \in \rm{lcc} (S', (\Theta_{S'})_{\leq 1}, \mathbf{Q})$. Now, let $\hat{X'}$ denote the formal completion of $X'$ along $S'$ and $T'= \sum T'_i$ the (possibly formal) separatrices meeting $S'$, then it follows from the proof of \cite[Lemma 8.9]{Spi} that $K_{S'}+(\Theta_{S'})_{\leq 1}+\mathbf{Q}_{S'}=(K_{\hat{X'}}+\D'+S'+T'+\M_{\hat{X'}})|_{S'}$. Again applying inversion of adjunction, we get $p' \in \rm{lcc} (\hat{X'}, \D'+S'+T',\M)$. By log smoothness, it follows that $p'$ is an lcc of $(X', \mcF',\D', \M)$ as required.\\
        Finally, let $I$ be the set of coefficients of $\D$ and all $\mu_j$'s such that $\M_{X'}=\sum \mu_j M_j$, where $\M$ descends on $X'$. Then by construction of $(\Theta_{S^{\nu}})_{\leq 1}$ and by \cite[Prop 4.9]{BZ}, the coefficients of $(\Theta_{S^{\nu}})_{\leq 1}=\pi_*((\Theta_{S'})_{\leq 1})$ are in the DCC set $D(I)$. Applying the same proposition again we see that the coefficients of $\lbrace\Theta_{C^{\nu}}\rbrace$ are in the DCC set $D(D(I))$.

    \end{proof}

\begin{remark}\label{adjn diff}
In the notation of the above proof, note that the effective divisor $(\Theta_{D^\nu})_{\leq 1}-\Theta'_{D^\nu}$ is supported on $h_{D'}(\sum_j(b_{1j}-a_{1j})E_{1j}|_{D'})$.
\end{remark}



    

\begin{lemma}(Stability of dlt gfqs under MMP) Let $(X,\mcF, \D,\M)/U$ be a rank two dlt gfq on a normal projective threefold $X$ equipped with a projective morphism $\pi:X\rightarrow U$ and $\phi:(X, \mcF,\D,\M) \dashrightarrow (X', \mcF', \D',\M)$ a $(K_\mcF+\D+\M_X)$-divisorial contraction or flip over $U$. Then $(X', \mcF', \D',\M)$ is also dlt. \end{lemma}
\begin{proof}
Indeed, we can argue similarly to \cite[Lemma 3.11]{CS}. Start with a foliated log resolution $g: Y \to X$ extracting only divisors $E$ with $a(E, \mcF,\D,\M) > - \epsilon(E)$ such that $\M$ descends to $Y$. We will show that $(X', \mcF', \D',\M)$ also admits such a foliated log resolution. Let $\overline{Y} \to Y \times X'$ denote the normalization of closure of the graph of $\phi \circ g$, $\mcG$ the induced foliation on $\overline{Y}$, $p: \overline{Y}\to Y$, $f:\overline{Y}\to X$ the induced morphisms and $F:= \sum F_i$ the reduced sum of all $f$-exceptional divisors. Let $h: Y' \to \overline{Y}$ be a foliated log resolution of $(\overline{Y}, \mcG, f_*^{-1}B+F)$ which is an isomorphism along the foliated log smooth locus of $(\overline{Y}, \mcG, f_*^{-1}B+F)$; note that $\M_{Y'}= g^*\M_{\overline{Y}}$. We claim that the induced morphism $g':Y' \to X'$ extracts only divisors $E'$ with $a(E', \mcF', B',\M) >-\epsilon(E')$.\\

Let $E'$ be a $g'$-exceptional divisor and $W:=c_{X'}(E')$. We claim that $a(E', X',\mcF',\D',\M)>-\epsilon(E')$. Let $\Sigma \subset X'$ be the flipped locus (if $\phi$ is a flip over $U$) or $\Sigma := \phi(\Ex \phi)$ (if $\phi$ is a divisorial contraction over $U$). Suppose $W \subset \Sigma$. Then by \cite[Lemma 3.38]{KM98}, $a(E', \mcF, \D,\M) < a(E', \mcF',\D',\M)$, while $a(E', \mcF, \D,\M) \geq -\epsilon(E')$ (since $(X,\mcF,\D,\M)$ is F-lc), thus $a(E', \mcF',\D',\M)> -\epsilon(E')$. Suppose $W \not \subset \Sigma$. Let $E'_Y$ denote the center of $E'$ on $Y$. 
Suppose, if possible, that $g$ is an isomorphism at the generic point of $E'_Y$. Then $Y$, $X$, $X'$, $\overline{Y}$ are all foliated log smooth and isomorphic to each other at the generic point of the center of $E'$. But $E'$ is then $h$-exceptional; this contradicts the choice of $h$. We thus infer that $X$ is not foliated log smooth at the generic point of $c_X(E')$. Now, let $l:=g^{-1}(c_X(E')) \subset Y$. Then $p:\overline{Y} \to Y$ is an isomorphism over the generic point of $l$. In particular, $\overline{Y}$ being foliated log smooth at the generic point of $\overline{l}:=p_*^{-1}l$, if $\overline{l}$ is not a divisor, it is not extracted by $h$. This is a contradiction. Thus $E'_Y$ is a $g$-exceptional divisor. Hence we have $-\epsilon(E')< a(E', \mcF, \D, \M) \leq a(E', \mcF', \D', \M)$. This proves our claim, thereby showing that $(X',\mcF', \D', \M)$ is a dlt gfq.
 
\end{proof}

\begin{lemma} (dlt modification of lc gfqs) Let $(X,\mcF, \D,\M)/U$ be a corank one lc gfq, where $X$ is a normal projective threefold equipped with a projective morphism $\pi:X\rightarrow U$. Then there exists a birational morphism $g: X' \to X$ from a $\mbQ$-factorial normal projective threefold with klt singularities such that letting $\mcG$ denote the pulled back foliation on $X'$ and $\D':=g_*^{-1}\D+\sum \epsilon(E_i)E_i$ (the sum runs over all $\pi$-exceptional divisors), $(X',\mcG, \D',\M)$ is dlt gfq and $K_\mcG+\D'+\M_{X'}=g^*(K_\mcF+\D+\M_X)$.\end{lemma}

\begin{proof}The arguments are similar to those of \cite[Theorem 8.1]{CS}. Let $g: W \to X$ be a foliated log resolution of $(X, \mcF,\D,\M)$ such that $\M$ descends to $W$, denote by $\mcG$ the induced foliation on $W$, let $\D_W$ be defined by $K_\mcG+\D_W+\M_W=g^*(K_\mcF+B+\M_X)$ and $\overline\D:=g_*^{-1}\D+\sum \epsilon(E_i)E_i$ (the sum runs over all $g$-exceptional divisors). Note that $\overline{\D}\geq \D_W$, and $(W,\mcG, \overline{\D},\M)$ is dlt. Let $\phi: W \dashrightarrow W'$ be a $(K_\mcG+\overline{\D}+\M_W)$-MMP over $X$. Observe that $K_\mcG+\overline{\D}+\M_W \equiv_X \overline{\D}-\D_W$. By the above remark, the induced gfq $(X',\mcG',\overline{\D}',\M)$ on $W'$ is dlt. It follows by an application of the negativity lemma that $\overline{\D}'=\D_W'$.
    
\end{proof}

Termination of flips for dlt gfqs does not follow directly from the lc case and requires considerable extra work. We first show special termination for dlt gfqs. The proofs proceeds by induction on the dimension of the lc centers. Then we use the structure of log terminal flips \cite[Lemma 2.8]{SS22} to reduce to the termination of flips for generalized pairs.
\begin{theorem}\label{gdlt-term}
    Starting from a dlt gfq $(X,\mcF, \D, \M)/U$ where $X$ is a $\mbQ$-factorial normal projective threefold equipped with a projective morphism $\pi:X\rightarrow U$ such that there exists $ 0 \leq B \leq \D$ such that $(X, B,\M)$ is gklt, there exists no infinite sequence of flips over $U$.
\end{theorem}
\begin{proof}
    Let $(X, \mcF, \D, \M)=:(X_1,\mcF_1, \D_1, \M) \dashrightarrow (X_2,\mcF_2, \D_2, \M) \dashrightarrow \cdots $ be an infinite sequence of flips. Let $\phi_i: (X_i,  \mcF_i, \D_i, \M) \dashrightarrow (X_{i+1}, \mcF_{i+1},\D_{i+1}, \M) $ be the $i$-th flip corresponding to the flipping contraction of an extremal ray $R_i \subset \overline{NE}(X_i/U)$. Let $C_i \subset X_i$ be a curve tangent to $\mcF_i$ such that $R_i= \mbR_+[C_i]$ (it exists thanks to Lemma \ref{dlt-lc} and Theorem \ref{lccone}).\\
    
    {\bf Step 1:} {\it After finitely many flips, \begin{enumerate}
        \item the flipping locus is disjoint from lc centers of $(X_i, \mcF_i, \D_i, \M)$ transverse to the foliation and zero dimensional generalized lc centers,
        \item no lc centers of $(X_i, \mcF_i, \D_i, \M)$ are contained in the flipping locus.\end{enumerate}} \vspace{0.5 mm}
    
    Let $S_i$ be a generalized lc center of $(X_i,  \mcF_i, \D_i, \M)$, which is transverse to the foliation $\mathcal{F}_i$. If $E$ is a geometric valuation over $X_i$ having center on $X_i$ equal to $S_i$, then by \cite[Lemma 3.11]{Spi}, we have that $a(E,X_i,\D_i,\M)=a(E,\mcF_i,\D_i,\M)=-1$. As $(X_i,\D_i,\M)$ is log smooth and $\M$ descends at the generic point of $S_i$ (by the definition of dlt gfq), $S_i$ is a strata of $\lfloor \D_i\rfloor$. 
    Thus it is enough to show that the flipping locus is eventually disjoint from $\lfloor\D_i\rfloor$ for $i>>0$ and we can assume $\dim S_i=2$. Let $C_i \subset X_i$ denote the flipping curve and $C_{i+1} \subset X_{i+1}$ denote the flipped curve. Then $C_{i+1}$ is tangent to $\mcF_{i+1}$. Indeed since $X_i$ is $\mbQ$-factorial, $(X_i, \mcF_i, \D_i, \M)$ dlt implies that so is $(X_i, \mcF_i, \D_i)$ by an application of negativity lemma. \cite[Theorem 11.3]{CS} then implies that $\mcF_i$ has non dicritical singularities and hence so does the induced foliation on the base of the flipping contraction (see also \cite[Lemma 3.31]{CS}). Thus $C_{i+1}$ is tangent to $\mcF_{i+1}$, because otherwise, contracting it would create a dicritical singularity of the base of the flipping contraction. We claim that $C_{i+1}$ is not contained in any component of $\lfloor \D_{i+1}\rfloor$. Indeed, if $C_{i+1} \subset S'_{i+1}$ for some component $S'_{i+1}$ of $\lfloor\D_{i+1}\rfloor$, then $C_{i+1}$  is invariant with respect to the restricted foliation $\mcF_{S'_{i+1}}$. By Proposition \ref{adj}, $(K_{\mcF_{i+1}}+\D_{i+1}+\M_{X_{i+1}})|_{S'_{i+1}}= K_{\mcF_{S'_{i+1}}}+\Theta_{S'_{i+1}}+\N_{S'_{i+1}}$, where the latter is an lc gfq. As a result, the generalized discrepancy $a(C_{i+1}, \mcF_{S'_{i+1}}, \Theta_{S'_{i+1}}, \N)=0$. For the same reason, letting $S'_i \subset X_i$ denote the proper transform of $S'_{i+1}$, $a(C_{i+1}, \mcF_{S'_i}, \Theta_{S'_i}, \N) \geq 0$. Here $(\mcF_{S'_i}, \Theta_{S'_i}, \N)$ is obtained by applying adjunction on $S'_i$. But by the arguments of \cite[Lemma 3.38]{KM98}, we have that $a(C_{i+1}, \mcF_{S'_{i+1}}, \Theta_{S'_{i+1}}, \N)> a(C_{i+1}, \mcF_{S'_i}, \Theta_{S'_i}, \N)$. So we have a contradiction, thereby proving the claim. In words, each flip $\phi_i$ where $C_i$ is contained in some component $S_i$ of $\lfloor\D_i \rfloor$ drops the Picard number of $S_i$ and does not increase the Picard number of any component of $\lfloor\D_{i+1}\rfloor$. Thus, there can only be finitely many such flips. On the other hand, if $C_i$ intersects some component $S_i$ of $\lfloor\D_i \rfloor$, then $(S_i \cdot C_i) >0 $. If $S_{i+1}=\phi_i(S_i)$, then by Lemma \ref{dlt-lc} and Theorem \ref{CT}, $(S_{i+1} \cdot C_{i+1}) <0$, which is ruled out by the above arguments. \\
    
    Now, we show that after finitely many flips, no lcc of the dlt gfq $(X_i, \mcF_i, \D_i, \M)$ is contained in the flipping locus. This is essentially a consequence of the fact that for a dlt gfq $(X, \mcF, \D, \M)$, there are only finitely many lc centers of codimension atleast two not contained in $\rm{Supp} \lfloor\D \rfloor$, combined with the negativity lemma. Indeed, let $Z$ be a lc center of $(X_i, \mcF_i, \D_i,\M)$ of codimension atleast two such that $Z$ is not a stratum of $\lfloor \D \rfloor$. Let $\pi: (X', \mcF') \to (X_i, \mcF_i)$ be a foliated log resolution of $(X_i, \mcF_i, \D_i, \M)$ which only extracts klt places of $(X_i, \mcF_i, \D_i, \M)$ (in particular, $\pi$ is an isomorphism at the generic point of $Z$) and such that $\M$ descends to $X'$. Let $Z' \subset X'$ be the strict transform of $Z$ in $X'$. Write $K_{\mcF'}+ \D'+\M_{X'}= \pi^*(K_{\mcF_i}+\D_i+\M_{X_i})$. Then clearly $Z'$ is an lc center of $(X',\mcF',\D'_+)$, where $\D'_+$ denotes the positive part of $\D'$. Then it follows from the proof of \cite[Proposition 3.9]{CS} that either $Z'$ is a stratum of $\lfloor \D'_+\rfloor$, or that it is a stratum of $\rm{Sing} \mcF'$. But since $Z$ is not a strata of $\lfloor \D \rfloor$ and $\pi$ is an isomorphism at the generic point of $Z$, it follows that $Z$ is a stratum of $\rm{Sing} \mcF$. In particular, there are only finitely many possibilities for $Z$. Now, returning to our set up, if $Z$ is a lc center of $(X_i, \mcF_i, \D_i, \M)$ contained in the flipping locus, the negativity lemma implies that the discrepancy of any exceptional divisor centered over $Z$ increases after the flip. Since there can only be finitely many such lc centers of a dlt gfq, we conclude that after finitely many flips, no lc center is contained in the flipping locus.  
     Now, if the flipping locus intersects a zero-dimensional lc center of $(X_i, \mcF_i, \D_i, \M)$, then the lc center is contained in the flipping locus. Thus, after fintely many steps, this can not happen.\\
     
    {\bf Step 2:} {\it After finitely many flips, the flipping locus is disjoint from generalized lc centers of dimension one tangent to the foliation.}\\
    
    Let $C$ be a one dimensional lc center of the dlt gfq $(X_i,\mathcal{F}_i,\D_i,\M)$, tangent to $\mathcal{F}_i$. By Proposition \ref{adj} we have that $(K_{\mathcal{F}_i}+\D_i+\M_i)|_{C^\nu}=K_{C^\nu}+\theta_i+\N_i$. 
    By step $1$ we know that after finitely many flips, the flipping curves are disjoint from $\lfloor\Delta_i\rfloor$, and $0$-dimensional lc centers. First, we claim that after finitely many flips, each $\phi_i$ is an isomorphism at the points of $(\theta_i)_{>1}$. Indeed, by proposition \ref{adj} the points of $(\theta_i)_{>1}$ are generalized log canonical centers of $(X_i,\mcF_i,\D_i,\M)$. Hence, our claim follows from the previous step. Now we focus on $(\theta_i)_{\leq 1}$.  Let $(X_{i+1},\mathcal{F}_{i+1},\Delta_{i+1},\M)$ be the flip of $(X_i,\mathcal{F}_i.\Delta_i,\M)$, and as this flip is isomorphism at the generic point of $C$ by step $1$, let $C'$ be its strict transform in $X_{i+1}$, which is a one-dimensional lc center of the dlt gfq $(X_{i+1},\mathcal{F}_{i+1},\Delta_{i+1},\M)$. We again apply sub-adjunction to get $(K_{\mathcal{F}_{i+1}}+\D_{i+1}+M_{X_{i+1}})|_{C'^\nu}=K_{C'^\nu}+\theta_{i+1}+\N_{i+1}$, where $\theta_{i+1}\geq 0$. We claim that $\theta_{i}\geqslant\theta_{i+1}\geqslant 0$, and that strict inequality holds along the points supported on the intersection of $C$ and the flipping curve $C_i$. Indeed, let $X_i \xleftarrow{p}\hat{X} \xrightarrow{q}X_{i+1}$ be the normalization of closure of the graph of $\phi_i: X_i \dashrightarrow X_{i+1}$ and $\hat{C} \subset \hat{X}$ the strict transform of $C$. Note that $C^\nu \cong C'^\nu$ and that via this isomorphism, $\N_i$ and $\N_{i+1}$ are same as b-divisors. This follows from the construction of Proposition \ref{adj}(3). Then $p^*(K_{\mcF_i}+\D_i+\M_{X_i})= q^*(K_{\mcF_{i+1}}+\D_{i+1}+\M_{X_{i+1}})+E$ for some $E\geq 0$ which is exceptional over both $X_i$ and $X_{i+1}$. Restricting this, we get $p^*(K_{C^\nu}+\theta_i+\N_i)=q^*(K_{C'^\nu}+\theta_{i+1}+\N_{i+1})+E|_{\hat{C}^{\nu}}$, where $\hat{C}^{\nu}$ is the normalization of $\hat{C}$ (by abuse of notation, we still denote by $C^\nu \xleftarrow{p}\hat{C}^\nu \xrightarrow{q}C'^{\nu}$ the induced morphisms). Since the difference between $\theta_i$ and $\theta_{i+1}$ comes from $E|_{\hat{C}^\nu}$, the claim follows.\\
    
  Let $\theta_i'=(\theta_i)_{\leq 1}$ be the boundary such that $(C^{\nu},\theta_i',\N_i)$ is glc. By Proposition \ref{adj}$(3)$ and Step $1$ we know that after finitely many flips the flipping locus is disjoint from the support of $\theta_i-\theta'_i$. If any component of $\theta'_i$ is supported on the intersection of $C$ with a flipping curve, then the coefficient of that component strictly drops. By Proposition \ref{adj} we know that coefficients of $\theta'_i$ belong to a DCC set, hence after finitely many flips all one dimensional glc centers tangent to the foliation are disjoint from the flipping locus. \\
    
     {\bf Step 3:} {\it After finitely many flips, the flipping locus is disjoint from all generalized lc centers .}\\
     
     Let $S_i$ be the normalization of a two-dimensional lc center of the gfq $(X_i,\mathcal{F}_i.\Delta_i,\M)$ which is tangent to the foliation $\mcF_i$, and $S_{i+1}$ be the normalization of its strict transform in the flip $X_{i+1}$. By adjunction we have that $(K_{\mathcal{F}_i}+\Delta_i+\M_i)|_{S_{i}}=K_{S_i}+\beta_i+\N_i$. Suppose we have the following sequence of flips
     \begin{equation*}
\xymatrix{    
               (X_i, \mathcal{F}_i,\Delta_i,\M)\ar[dr]_{f_i} \ar@{-->}[rr] &     & (X_{i+1},\mathcal{F}_{i+1},\Delta_{i+1},\M)\ar[dl]^{f_0^+}\ar[dr]_{f_{i+1}} \ar@{-->}[rr]& & (X_{i+2},\mathcal{F}_{i+2},\Delta_{i+2},\M)\ar[dl]^{f_2^+} \cdots&\\
                                             &Z_i&                  & Z_{i+1}&
                                             }
 \end{equation*} 
 Restricting this diagram to $S_i$ we get the following diagram-
 \begin{equation*}
\xymatrix{    
               (S_i,\beta_i,\N)\ar[dr]_{f_i|_{S_i}} \ar@{-->}[rr] &     & (S_{i+1},\beta_{i+1},\N)\ar[dl]^{f_i^+|_{S_{i+1}}}\ar[dr]_{f_{i+1}|_{S_{i+1}}} \ar@{-->}[rr]& & (S_{i+2},\beta_{i+2},\N)\ar[dl]^{f_1^+|_{S_{i+2}}} \cdots&\\
                                             &T_i&                  & T_{i+1}&
                                             }
 \end{equation*} 
 where $T_i$ is the normalization of the image of $f_i|_{S_i}$. A priori, though $(S_i,\beta_i,\N)$ is not a generalized lc pair, by lemma \ref{adj}, $(S_i,\beta_i':=(\beta_i)_{\leq 1},\N)$ is generalized lc. We claim that for $i\gg 0$, $(\beta_i-\beta_i')\cdot C_i=0$. This follows from the fact that $\beta_i-\beta_i'$ is supported on codimension two lc centers of $(X_i, \mcF_i, \D_i,\M)$ contained in $S_i$ by Proposition \ref{adj}. By step $2$, $C_i$ is disjoint from such centers.
 
 
 This implies after finitely many flips $(K_{S_i}+\beta'_i+\N)\cdot C_i<0$, in particular we get the following sequence of \textit{ample small quasi-flips} \cite[Def 2.13]{LMT}:
 \begin{equation*}
\xymatrix{    
               (S_i,\beta'_i,\N)\ar[dr]_{f_i|_{S_i}} \ar@{-->}[rr] &     & (S_{i+1},\beta'_{i+1},\N)\ar[dl]^{f_i^+|_{S_{i+1}}}\ar[dr]_{f_{i+1}|_{S_{i+1}}} \ar@{-->}[rr]& & (S_{i+2},\beta'_{i+2},\N)\ar[dl]^{f_1^+|_{S_{i+2}}} \cdots&\\
                                             &T_i&                  & T_{i+1}&
                                             }
 \end{equation*}
    Now, by the construction of \cite[Lemma 3.2]{LMT}, there exists a generalized dlt modification $(\tilde{S}_j,\tilde{\beta}_j,N)\rightarrow (S_j,\beta_j',\N) $ for each $j$, such that $(\tilde{S}_j,\tilde{\beta}_j,N)\dashrightarrow (\tilde{S}_{j+1},\tilde{\beta}_{j+1},N)$ is a $(K_{\tilde{S}_j}+\tilde{\beta}_j+\N_{\tilde{S}_j})$-MMP over $T_j$. If $({S}_i,{\beta}_i,N)\dashrightarrow ({S}_{i+1},{\beta}_{i+1},N)$ is not an isomorphism for all $i>>0$, that would give rise to a $(\tilde{S}_j,\tilde{\beta}_j,N)$-MMP which does not terminate as in the proof of Theorem \ref{lcterm}, which is a contradiction. Hence, such sequence must be isomorphism after finitely many $i$, and the flipping locus is disjoint from $S_i$ for large $i$.\\

    We have shown that for $i \gg 0$, $\phi_i$ is disjoint from all lc centers of $(X_i, \mcF_i, \D_i, \M)$. Thus, it suffices to show that any sequence of flips $\phi_i$ for a log terminal gfq $(X_i, \mcF_i, \D_i, \M)$ terminates. The rest of the proof is similar to \cite[Theorem 2.1]{SS22} to which we refer for more details. By negativity lemma, $\mcF_i$ has terminal singularities at the generic point of the flipping curve $C_i$. We can find a unique $\mcF_i$-invariant surface $S$ in a small analytic neighbourhood $U$ of $C_i$ containing $C_i$; note that $S$ is analytically $\mbQ$-Cartier. Moreover, $S$ is the unique $\mcF_i$-invariant divisor meeting $C_i$. By \cite[Lemma 2.8]{SS22}, $S\cdot C_i=0$ and from \cite[Lemma 8.9]{Spi}, $(K_{\mcF_i}+\D_i+\M_{X_i})\cdot C_i \geq (K_{X_i}+\D_i+S+\M_{X_i})\cdot C_i$. In particular, each $\phi_i$ is a $(K_{X_i}+(1-\epsilon)\D_i+\M_{X_i})$-flip, where the latter is gklt by Lemma \ref{Fdlt-dlt}. Hence, the desired termination follows from \cite[Theorem 1.5]{HL}.

\end{proof}

\subsection{Minimal model program for lc gfqs}
In the earlier parts of this paper, we saw how the MMP for foliated lc triples follows from the MMP for foliated dlt triples. In a similar vein, in this section, we observe how the MMP for lc gfqs follows from MMP for dlt gfqs.

\begin{theorem} \label{glccone}
    Let $(X,\mathcal{F},\D,\M)/U$ be a rank two lc gfq where $X$ is a normal projective threefold equipped with a projective morphism $\pi: X \to U$. Then there exists a countable collection of rational curves $C_i$ on $X$ tangent to $\mcF$ such that.
    \begin{enumerate}
        \item $\overline{NE}(X/U)=\overline{NE}(X/U)_{(K_{\mathcal{F}}+\D+\M_X)\geqslant 0}+\sum \mbR_+\cdot [C_i]$
        \item $-6\leqslant (K_X+\D+\M_X)\cdot C_i<0$
        \item For any ample over $U$ divisor $H$, $(K_{\mathcal{F}}+\D+M_X+H)\cdot C_i\leqslant 0$ for all but finitely many $i$.
        \end{enumerate}
\begin{proof}
    The proof is similar in spirit to that of Theorem \ref{lccone} once we have the cone theorem for $\mbQ$-factorial dlt gfqs. \\

Let $(X',\mcF', \D', \M)$ be  $\mbQ$-factorial dlt gfq. For any ample $\mbR$-divisor $A$, since $K_{\mcF'}+ \D'+ \M_{X'}+A$ is $\mbR$-linearly equivalent over $U$ to a foliated lc triple by Lemma \ref{dlt-lc}, the relative cone theorem for $(X',\mcF', \D', \M)$ over $U$ follows from the relative F-lc cone theorem (Theorem \ref{lccone}) using similar arguments as in the proof of the cone theorem for dlt pairs; see for example \cite[Theorem 3.35]{KM98}.\\

Let $\pi: (X',\mcF', \D', \M) \to (X,\mcF, \D, \M)$ be a foliated dlt modification. Now we can deduce the cone theorem for $K_\mcF+\D+\M_X$ from the cone theorem for $K_{\mcF'}+\D'+\M_{X'}$ using the same arguments as in the proof of Theorem \ref{lccone}. 
\end{proof}
\end{theorem}

We will need the following lemma for proving the contraction and flip theorems for lc gfqs.

\begin{lemma}\label{Fdlt-dlt}
Let $(X, \mcF, \D, \M)$ be a $\mbQ$-factorial dlt gfq with $\lfloor \D \rfloor=0$. Then for any reduced $\mcF$-invariant divisor $D=\sum D_i$, and $ \epsilon >0$,  $(X, \D+(1-\epsilon)D,\M)$ is gklt.
    
\end{lemma}
\begin{proof}
    By definition of dlt gfq, there exists a log resolution $\pi:(X',\mcF')\to (X, \mcF)$ of $(X, \mcF, \D, \M)$ which extracts only klt places of $(X, \mcF, \D, \M)$ and to which $\M$ descends. Let $E_{i0}$ be the collection of $\pi$-exceptional invariant divisors,  $E_{j1}$ the $\pi$-exceptional non invariant ones, $D':=\pi_*^{-1}D$ and $\D':=\pi_*^{-1}\D$. Write $K_{\mcF'}+\D'+\sum E_{j1} +\M_{X'}=\pi^*(K_\mcF+\D+\M_X)+\sum a_{j1}E_{j1}+\sum a_{i0}E_{i0}$, where $a_{j1}, a_{i0}>0$ for all exceptional divisors $E_{j1}$ and $E_{i0}$. Similarly, we can write $K_{X'}+ \D'+D'+\sum E_{j1}+\sum E_{i0}+\M_{X'}=\pi^*(K_X+\D+D+\M_X)+\sum b_{j1}E_{j1}+\sum b_{i0}E_{i0}$. Then it follows from the arguments of \cite[Lemma 8.14]{Spi} that $b_{j1} \geq a_{j1}$ and $b_{i0} \geq a_{i0}$. In particular, $b_{j1}, b_{i0} >0$. Now writing the corresponding equation for $K_X+\D+(1-\epsilon)D+\M_X$, $K_{X'}+\D'+(1-\epsilon)D'+\sum E_{j1}+\sum E_{i0}+\M_{X'}=\pi^*(K_X+\D+(1-\epsilon)D+\M_X)+\sum b^\epsilon_{j1}E_{j1}+\sum b^\epsilon_{i0}E_{i0}$, we have $b^\epsilon_{j1}\geq b_{j1}$ and $b^\epsilon_{i0}\geq b_{i0}$. From this, it follows that $(X, \D+(1-\epsilon)D, \M)$ is gklt.
\end{proof}
We are now in a position to develop the MMP for lc gfqs.

\begin{theorem} \label{glcCT}
Let $(X,\mcF, \D, \M)/U$ a rank two lc gfq where $X$ is a projective threefold equipped with a projective morphism $\pi: X \to U$ such that $(X,B, \M)$ is gklt for some $0\leq B\leq \D$. Let $R \subset \overline{NE}(X/U)$ be a $(K_\mcF+\D+\M_X)$-negative exposed extremal ray. Then there exists a contraction $c_R: X \to Y$ over $U$ associated to $R$, where $Y$ is a normal projective variety of klt type. Moreover, $c_R$ satisfies the following properties:
    \begin{enumerate}
        \item If $L\in \Pic X$ is such that $L \equiv_Y 0 $, then there exists $M\in \Pic Y$ with $c_R^*M=L$,
        \item If $X$ is $\mbQ$-factorial and $c_R$ is a divisorial or Fano contraction, then $Y$ is also $\mbQ$-factorial.
    \end{enumerate}
    
\end{theorem}
\begin{proof}
    The arguments are parallel to those used in Theorem \ref{CT}, so we provide only a sketch of the proof. Let $\pi: (\overline{X}, \overline{\mcF}) \to (X, \mcF)$ be a $\mbQ$-factorial dlt modification of $(X, \mcF, \D, \M)$ which factors through a small $\mbQ$-factorialization $h:\tilde{X} \to X$ of $X$ . As before, write $K_{\overline{\mcF}}+\overline{\D}+\M_{\overline{X}}=\pi^*(K_{\mcF}+\D+\M_X)$, $K_{\overline{X}}+\overline{B}+E+\M_{\overline{X}}=\pi^*(K_X+B+\M_X)+F$, where $\overline{\D}:=\pi_*^{-1}\D+\sum \epsilon(E_i)E_i$, $\overline{B}:=\pi_*^{-1}B$, $E, F \geq 0$ are $\pi$-exceptional divisors which contain all exceptional divisors of $h$. Now, $(\overline{X}, \overline{B}+E, \M)$ is gklt by Lemma \ref{Fdlt-dlt}. Let $H_R$ be a supporting Cartier divisor of $R$. If $H_R$ is not big, by an application of bend and break as in Theorem \ref{CT}, it follows that $(K_X+B+\M_X) \cdot R <0$, so we can use the contraction theorem for klt pairs to get $c_R$. Otherwise, we may assume $(K_X+B+\M_X)\cdot R>0$. In this case, we run a $(K_{\overline{\mcF}}+\overline{\D}+\M_{\overline{X}})$-MMP with scaling of $\overline{A}:=\pi^*A$. Let $i_0$ be the first step of this for which the nef threshold of $\overline{A}_{i0}$ with respect to $K_{\overline{\mcF}_{i0}}+\overline{\D}_{i0}+\M_{\overline{X}_{i0}}$, say $\lambda$ becomes less that $1$ and $\phi: \overline{X} \dashrightarrow \overline{X}'$ the induced birational contraction. Let $\overline{X}':=\overline{X}_{i0}$. Pick $\lambda < \lambda' <1$ such that $K_{\mcF}+\D+\lambda'A$ is big and $s>0$ small enough such that if $K:=(K_{\overline{\mcF}}+\overline{\D}+\M_{\overline{X}}+\lambda'\overline{A})+s(K_{\overline{X}}+\overline{B}+E+\M_{\overline{X}})$, then $K$ is big and $\phi $ is $K$-negative. Now, we run a $K':=\phi_*K$-MMP $\psi:\overline{X}' \dashrightarrow \overline{X}''$ which is $(K_{\overline{X}'}+\overline{B}'+E'+\M_{\overline{X}'})$-negative, $(K_{\overline{\mcF}'}+\overline{\D}'+\M_{\overline{X}'}+\lambda'\overline{A}')$-trivial and $\overline{H}_R$-trivial. Then $\psi \circ \phi$ contracts the divisorial part of $\pi^{-1} \rm{loc}R \cup \rm{Ex} \pi$. We may manufacture another supporting Cartier divisor $\tilde{H}_R$ of $R$ and as in the proof of Theorem \ref{CT} such that letting $f: X \dashrightarrow \overline{X}''$ the induced map, $f_*\tilde{H}_R-(K_{\overline{X}''}+\overline{B}''+\M_{\overline{X}''})$ is nef and big. Now, $(\overline{X}'', \overline{B}'', \M)$ is gklt by Lemma \ref{Fdlt-dlt} and since $\psi$ is $(K_{\overline{X}'}+\overline{B}'+E'+\M_{\overline{X}'})$-negative. Thus $f_*\tilde{H}_R$ is then semiample by basepoint free theorem (see \cite[Theorem 2]{IJM} for a more general result). In particular, so is $\tilde{H}_R$, giving us the contraction $c_R: X \to Y$. The descent of numerically trivial divisors follows same arguments as Theorem \ref{CT}.
\end{proof}

\begin{theorem}\label{glcflips}
     Let $(X, \mcF, \Delta, \M)/U$ be a projective rank two lc gfq such that $(X, B, \M)$ is klt for some $B \leq \Delta$, and $\pi:X\rightarrow U$ be a projective morphism of projective varieties and $\dim X =3$. Let $f: X \to Z $ be the contraction of a $(K_\mcF+\Delta+\M )$-negative extremal ray over $U$. Then the log canonical model $(X^+, \mcF^+, \Delta^+, \M)$ of $(K_\mcF+\Delta+\M_X)$ over $Z$ exists. Moreover, letting $B^+$ denote the strict transform of $B$ on $X^+$, $(X^+, B^+,\M)$ is klt and $(X^+,\mcF^+, \Delta^+,\M)$ is lc.
\end{theorem}
\begin{proof}
    Since we have the dlt modification for lc gfqs and the full MMP for dlt gfqs, we can argue as in Theorem \ref{lcflip}. Let $g: (\overline{X},\overline{\mcF}, \overline{\D}, \M)\to (X, \mcF, \D, \M)$ be a dlt modification of $(X, \mcF, \D, \M)$. Let $K_{\overline{\mcF}}+\overline{\D}+\M_{\overline{X}}=\pi^*(K_{\mcF}+\D+\M_X)$ and $K_{\overline{X}}+\overline{B}+E+\M_{\overline{X}}=\pi^*(K_X+B+\M_X)+F$ (notation as in proof of Theorem \ref{glcCT}). 
    Let $\phi: \overline{X}\dashrightarrow \overline{X}'$ be a $(K_{\overline{\mcF}}+\overline{\D}+\M_X)$-MMP over $Z$ and $\psi:\overline{X}' \dashrightarrow \overline{X}''$ be a $K:=K_{\overline{\mcF}}+\overline{\D}+\M_{\overline{X}}+s(K_{\overline{X}}+\overline{B}+E+\M_{\overline{X}})$-MMP over $Z$. Then for small enough $s>0$, $\psi$ is $(K_{\overline{\mcF}'}+\overline{\D}'+\M_{\overline{X}'})$-trivial and $\psi \circ \phi$ contracts the divisorial part of $\Ex g$ and $K'':= (\psi\circ \phi)_*K$ is semiample over $Z$ by the basepoint free theorem. This forces $K_{\overline{\mcF}''}+\overline{\D}''+\M_{\overline{X}''}$ to be the same. Then the flip $(X, \mcF, \D, \M) \dashrightarrow (X^+, \mcF^+, \D^+, \M)$ is given by its semiample fibration. $(\overline{X}'', \overline{B}'',\M)$ is gklt by Lemma \ref{Fdlt-dlt} and the semiample fibration is trivial with respect to this gklt pair.
\end{proof}

We also have the following corollary to Lemma \ref{gdlt-term}.

\begin{corollary}\label{glcterm} Starting from a projective lc gfq $(X,\mcF, \D, \M)/U$ such that $(X,B, \M)$ is gklt for some $B \leq \D$, there exists no infinite sequence of $(K_\mcF+\D+\M_X)$-MMP over $U$.
   \begin{proof}
Since we have termination of MMP for dlt gfqs, the arguments are parallel to those used in Theorem \ref{lcterm}. Let\begin{equation*}
\xymatrix{    
               (X_0, \mathcal{F}_0,\Delta_0, \M)\ar[dr]_{f_0} \ar@{-->}[rr] &     & (X_1,\mathcal{F}_1,\Delta_1, \M)\ar[dl]^{f_0^+}\ar[dr]_{f_1} \ar@{-->}[rr]& & (X_2,\mathcal{F}_2,\Delta_2, \M)\ar[dl]^{f_1^+} \cdots&\\
                                             &Z_0&                  & Z_1&
                                             }
 \end{equation*} 
 be a sequence of lc gfq MMP starting from $(X_0, \mcF_0, \D_0,\M):= (X, \mcF, \D, \M)$. We go to a dlt modification $(\overline{X_0}, \overline{\mcF_0},\overline{\Delta_0}, \M)$ of $(X_0,\mcF_0,\Delta_0,\M)$ and run a $(K_{\overline{\mcF_0}}+\overline{\Delta_0}+\M_{\overline{X_0}})$-MMP $\phi_0:\overline{X_0} \dashrightarrow \overline{X_0}'$ followed by a $K:=K_{\overline{\mathcal{F}_0}'}+\overline{\Delta_0}'+\M_{\overline{X_0}'}+s(K_{\overline{X}_0'}+\overline{B}'+E^{'}+\M_{\overline{X_0}'})$-MMP $\psi_0: \overline{X_0}' \dashrightarrow \overline{X_0}''$ over $Z_0$ for $s>0$ small as above. As $K_{\overline{\mathcal{F}_0}''}+\overline{\Delta_0}''+\M_{\overline{X_0}''}$ is semi-ample over $Z_0$, so is $K_{\overline{\mathcal{F}_0}'}+\overline{\Delta_0}'+\M_{\overline{X_0}'}$. $(X_1, \mathcal{F}_1,\Delta_1,\M)$ being the ample model, we have a morphism $\pi_1:\overline{X}_0'\rightarrow X_1$, given by the semiample fibration of $K_{\overline{\mathcal{F}_0}'}+\overline{\Delta_0}'+\M_{\overline{X_0}'}$ over $Z_0$.  Since $(\overline{X_0}', \overline{\mcF_0}', \overline{\Delta_0}',\M)$ is dlt gfq, we see that $\pi_1$ is a dlt modification of $(X_1, \mcF_1, \Delta_1,\M)$. Now we can again repeat the process starting from $(X_1, \mcF_1, \Delta_1,\M)$ over $Z_1$. Hence we get a sequence of foliated flips and divisorial contractions over $U$, in particular over Spec $\mbC$ starting from a dlt gfq $(\overline{X_0},\overline{\mathcal{F}_0},\overline{X}_0,\overline{\Delta_0})$ which must terminate by Theorem \ref{gdlt-term}. Hence the starting sequence of log canonical flips cannot be infinite.
 
   \end{proof} 
\end{corollary}

\subsection{A basepoint free theorem}

As an application of the MMP for lc gfqs developed earlier, we prove the following basepoint free theorem for foliations. Our approach is somewhat different from that of the F-dlt case treated in \cite{CS}.

\begin{theorem} \label{bpft}
    Let $\pi:X\rightarrow U$ be a projective morphism of normal projective varieties and let $(X,\mathcal{F},\Delta, \M)/U$ be a corank one lc gfq such that $(X, B, \M)$ is gklt for  some $0 \leq B \leq \Delta$ and $\dim X=3$. If $K_{\mathcal{F}}+\Delta+\M_X+A$ is nef over $U$ for some ample over $U$ $\mbR$-divisor $A$ on $X$, then $K_{\mathcal{F}}+B+\M_X+A$ is semiample over $U$.
\end{theorem}
\begin{proof}
 We divide the proof into two cases:
\begin{enumerate}
\item Suppose $K_{\mathcal{F}}+\Delta+\M_X+A$ is big over $U$. In this case, we use some ideas from the proof of \cite[Theorem 1.6]{MMPlcintegrable}. Let $L:=K_\mcF+\Delta+\M_X+A$ and assume $L$ is NQC, i.e. there exist positive real numbers $a_i$ and nef over $U$ $\mbQ$-Cartier $\mbQ$-divisors $L_i$ such that $L= \sum_i a_i L_i$. We can choose $0< e \ll 1$ such that if $\hat{A}:= A+e(K_\mcF+\Delta+\M_X)-e(K_X+B+\M_X)$, then $\hat{A}$ is ample. Letting $K:=(1-e)(K_\mcF+\Delta+\M_X)+e(K_X+B+\M_X)$, note that $L=K+\hat{A}$. Since $K_\mcF+\Delta+\M_X+A$ is big over $U$, so is $K_\mcF+\Delta+\M_X+\hat{A}$. Note that for any $l \in \mathbb{N}$, $(X,\mcF, \Delta,\M+\hat{A}+lL)$ is a lc gfq with NQC moduli part. Let $\phi: X \dashrightarrow Y$ be a $(K_\mcF+\Delta+\M+\hat{A}+lL)$-MMP, where $l \in \mbN$; it follows from the length estimate for ($K_\mcF+\D+\M+\hat{A})$-negative extremal rays that we can choose $l_0 \in \mbN$ (depending on $a_i$ and the Cartier indices of the $L_i$) such that for all $l \geq l_0$, $\phi$ is $L$-trivial. Since $(X,B,\M+\hat{A}+lL)$ is gklt, so is $(Y, B_Y,\M+\hat{A})$. Next, observe that
\begin{center}
$(l+1-el)L=K+\hat{A}+(l-el)L = e(K_X+B+\M_X+\hat{A})+(1-e)(K_\mcF+\Delta+\M_X+\hat{A}+lL)$.
\end{center} 
On $Y$, this gives 
\begin{center} 
$K_Y+B_Y+\M_Y+\hat{A}_Y+\frac{1-e}{e}(K_{\mcF_Y}+\Delta_Y+\M_Y+\hat{A}_Y+lL_Y)= \frac{l+1-el}{e}L_Y$.
\end{center}
 By \cite[Lemma 4.4]{BZ}, the gklt pair $(Y,  B_Y,\M+\hat{A}+\frac{1-e}{e}(K_\mcF +\Delta+\M_X+\hat{A}+lL))$ has a good minimal model. This implies that $L_Y$ is semiample. Since $\phi$ is $L$-trivial, $L$ is also semiample.\\

 Thus we have proved that if $L$ is NQC and big, then it is semiample. Now we show that $L$ is always NQC to conclude the big case. First, replacing the lc gfq $(X,\mcF, \D,\M)$ with $(X,\mcF, \D,\M+\frac{1}{2}A)$, the number of $(K_\mcF+\D+\M_X)$-negative extremal rays is finite; say they are $R_1 , \cdots, R_k$. Let $F:=L^{\perp} \cap \overline{NE}(X/U)$ and note that $L$-trivial extremal rays are $(K_\mcF+\D+\M_X)$-negative. Thus $F$ is spanned by a subset of $R_1, \cdots, R_k$, say $R_1, \cdots, R_m$ with $m \leq k$. Let $V$ be the smallest affine subspace of $WDiv_\mbR(X)$ defined over $\mbQ$ containing $L$. Let $\mathcal{C}:= R_1^{\perp} \cap \cdots \cap R_m^{\perp} \cap V$; it is a rational polyhedron. Then we claim that $L$ can be written as $L=\sum r_iD_i$, where $r_i$ are positive real numbers and $D_i \in \mathcal{C}$ are nef $\mbQ$-Cartier $\mbQ$-divisors over $U$. Indeed, since $(L \cdot R_j)>0$ for $m+1 \leq j \leq k$, choosing $D_i \in \mathcal{C}$ rational and sufficiently close to $L$, $(D_i \cdot R_j)>0$ for $m+1 \leq j \leq k$ and also $D_i-(K_{\mathcal{F}}+\Delta+\M_X)$ is ample over $U$. Thus, if $R \subset \overline{NE}(X/U)$ is such that $(K_\mcF +\D+\M_X)\cdot R \geq 0$, then $(D_i \cdot R) >0$. Finally, $(D_i \cdot R_i)=0$ for $1 \leq j \leq m$. This shows that the $D_i$ are nef over $U$. 
\vspace{5mm}
\item Suppose $K_\mcF+\Delta+\M_X+A$ is not big over $U$. Then $K_\mcF+\Delta+\M_X$ is not pseudoeffective over $U$. Run a $(K_\mcF+\Delta+\M_X)$-MMP over $U$ with scaling of $A$, say $X \dashrightarrow X_1 \dashrightarrow \cdots X_i \dashrightarrow \cdots X'$. Then $X'$ admits a $(K_{\mcF'}+\Delta'+\M_{X'})$-Mori fiber space structure over $U$, say $X' \xrightarrow{f}S$ and $\mcF'$ descends to a foliation on $S$. Assume $\dim S=2$. Let $\lambda_i:=\rm{inf}\{t|K_{\mcF_i}+\D_i+\M_{X_i}+tA_i$ is nef$\}$. If $\lambda_i<1$ for some $i$, choose $i$ to be the smallest index for which this happens; note that the MMP $X \dashrightarrow X_i$ is $(K_\mcF+\Delta+\M+A)$-trivial. Then $K_{\mcF_i}+\Delta_i+\M_{X_i}+A_i$ is nef and big over $U$, which forces $K_\mcF+\Delta+\M_X+A$ to be the same, a contradiction. We conclude $\lambda_i=1$ for all $i$, hence the entire MMP is $(K_\mcF+\Delta+\M+A)$-trivial. In particular $K_{\mcF'}+\Delta'+\M_{X'}+A_{X'} \sim _{\mbR, S}0$ (here we consider $(X,\mcF,\D, \M+A)$ as an lc gfq). We now claim that there exists a NQC b-nef divisor $\M'_{X'}$ and an ample $\mbR$-divisor $H'$ such that $\M_{X'}+A_{X'} \sim_\mbR \M'_{X'}+H'$. For this, we can argue as follows: let $\phi_1: X \dashrightarrow X_1$ be the first step of the MMP. Let $H_1$ be an ample over $U$ $\mbR$-divisor on $X_1$ and $H_X:=\phi_{1*}^{-1}H_{1}$. We can choose $\epsilon >0$ such that $A- \epsilon H_X \sim _\mbR \Theta$ is ample over $U$. Then, since $\phi_1$ is $(K_\mcF+\Delta+\M_X+A) \sim_\mbR (K_\mcF+\Delta+\M_X+\Theta_{X_1}+\epsilon H_X )$-trivial, $(X_1, \mcF_1, \Delta_1, \M+\Theta +\epsilon H_1)$ is lc gfq. Now note that $K_{\mcF_1}+\Delta_1+\M_{X_1}+A_{X_1} \sim_{\mbR,U} K_{\mcF_1}+\Delta_1+\M_{X_1}+\Theta_{X_1}+\epsilon H_1$. This shows that our claim holds for $X_1$ and by repeating this argument, we get the claim for $X'$. Let $\mcG$ denote the induced foliation on $S$. By \cite[Theorem $2.3.2$]{CHLX}, there exists an lc gfq structure $(S,\mcG, \D_S,\N)$ such that $K_{\mcF'}+\D'+ \M'_{X'}+H' \sim_{\mbR,U} f^*(K_{\mcG}+\D_S+\N_S)$. Let $A_S$ be an ample $\mbR$-divisor on $S$ such that $H'-f^*A_S$ is ample over $U$. Consider the modified canonical bundle formula $K_{\mcF'}+\D'+ \M'_{X'}+H'-f^*A_S \sim_{\mbR,U}K_{\mcG}+\D_S+\N'_S$ where $(S,\mcG, \D_S, \N')$ is lc gfq with NQC b-nef divisor $N'$. Comparing this with the previous canonical bundle formula gives $\N_S \sim_{\mbR,U} \N'_S+A_S$. Thus $K_{\mcG}+\D_S+\N_S \sim_{\mbR,U} K_{\mcG}+\D_S+\N'_S+A_S$. Since $(X,B,\M)$ is gklt, there exists a gklt structure $(S, B_S, M^S)$ on $S$. Let $\mu:S' \to S$ be a small modification which is $\mbQ$-factorial klt. We argue similarly to the threefold case treated above. First, assume $L_S:= K_\mcG+\Delta_S+\N'_S+A_S$ is big over $U$. Let $\hat{A}_S:=A_S+e(K_\mcG+\D_S+\N'_S)-e(K_S+B_S+\M^S_S)$; it is ample if $0<e \ll 1$, $K_\mcG+\D_S+\N'_S+\hat{A}_S$ is big. Letting $K':=(1-e)(K_\mcG+\D_S+\N'_S)+e(K_S+B_S+\M_S^S)$ we have $L_S=K'+\hat{A}_S$. Let $S' \dashrightarrow S''$ be a $(K_{\mcG'}+\D_{S'}+\N'_{S'}+\hat{A}_{S'}+lL_{S'})$-MMP over $U$; note that it exists by \cite[Theorem 2.8]{SS23} and for $l \gg 0$, it is $L_{S'}$-trivial. Then, as in the threefold case, \begin{center}
    $K_{S''}+B_{S''}+\M^S_{S''}+\hat{A}_{S''}+\frac{1-e}{e}(K_{\mcG''}+\Delta_{S''}+\N'_{S''}+\hat{A}_{S''}+lL_{S''})= \frac{l+1-el}{e}L_{S''}$ 
    \end{center} is semiample over $U$ by \cite{BZ} and we are done. We are left to deal with the case $\dim S=1$. Then $\mcF''$ is induced by $X'' \to S$, so we can use \cite[Theorem 1.6]{MMPlcintegrable} to conclude.\\

    Now in case $L_S$ is not big over $U$, we run a $(K_{\mcG'}+\D_{S'}+\N'_{S'})$-MMP over $U$ with scaling of $A_{S'}$; as in the above case, it is $(K_{\mcG'}+\D_{S'}+\N'_{S'}+A_{S'})$-trivial and ends with a Mori fiber space $g:S''\to C$ such that $K_{\mcG''}+\D_{S''}+\N'_{S''}+A_{S''} \sim_{\mbR, C}0$. In this case, $g:S' \to C$ induces $\mcG'$. So we can use \cite[Theorem 2.3.1]{M7'} to conclude.

\end{enumerate}
\end{proof}

\begin{corollary}\label{gmm}
    Let $\pi:X\rightarrow U$ be a projective morphism of normal projective varieties and $\mcF$ be a corank one foliation on the normal projective threefold $X$. Suppose $(X,\mcF, \D, \M)$ is a lc gfq such that $(X,B,\M)$ is gklt for some $0 \leq B \leq \D$. Let $A$ be an $\pi$-ample $\mbR$-divisor on $X$ such that $K_\mcF+\D+\M_X+A$ is pseudoeffective. Then the lc gfq $(\mcF, \D, \M+A)$ has a good minimal model over $U$.

    \begin{proof} We argue along the lines of \cite[Corollary 5.3]{Chaudas}. Let $(X, \mcF,\D, \M+A) \dashrightarrow (X_1, \mcF_1, \D_1, \M+A) \dashrightarrow \cdots$ be a $(K_\mcF+B+\M+A)$-MMP over $U$. We claim the property of the moduli part containing an ample divisor carries through the MMP. Indeed, let $\phi_1:X \dashrightarrow X_1$ be the first step of the MMP. Choose $\pi$-ample $\mbR$-divisor $H_1$ on $X_1$ such that if $H_X:=\phi_{1*}^{-1}H_1$, there exists $C \sim_{\mbR,U} A-H_X$ which is ample. $\phi_1$ is clearly also a $(K_\mcF+\D+\M+\A+\epsilon C)$-MMP if $\epsilon >0$ is sufficiently small. Moreover, $(X,\mcF, \D, \M+\A+\epsilon \C)$ is lc and so is $(X,\mcF, \D, \M+(1-\epsilon)\A+\epsilon \C)$. Clearly, $\phi_1$ is also a $(K_\mcF+\D+\M+(1-\epsilon)\A +\epsilon \C)$-MMP as well. Thus, $(X_1,\mcF_1, \D_1, \M+(1-\epsilon) \A+\epsilon \C)$ is lc. This implies that $(X_1,\mcF_1, \D_1, \M+(1-\epsilon)\A+ \epsilon \C +\epsilon H_1)$ is lc. Now note that 
\begin{center} 
   $ K_{\mcF_1}+\D_1+\M_{X_1}+(1-\epsilon)\A_{X_1}+\epsilon \C_{X_1}+\epsilon H_1 \sim _{\mbQ} K_{\mcF_1}+\D_1+\M_{X_1}+\A_{X_1}= 
    \phi_{1*}(K_\mcF+\D+\M_X+\A)$.
\end{center}
Arguing similarly on subsequent steps proves the claim. Let $\phi: X \dashrightarrow X'$ denote the full MMP. Let $(X',\mcF', \D', \N'+H')$ denote the induced lc gfq with $H'$ ample over $U$. Then $K_{\mcF'}+\D'+\N'+H'$ is semiample over $U$ by Theorem \ref{bpft}.
    \end{proof}
\end{corollary}

\section{Flop connections between minimal models}
As another application of the foliated log canonical MMP, in this section, we show how to connect any two minimal models of an lc foliated triple $(X,\mcF, \Delta)$ by a sequence of flops. In case $\Delta=0$, $\mcF$ has F-dlt singularities and $K_\mcF$ is big, this was proved in \cite{VJ}. We start this section by providing an example of non-isomorphic $\mbQ$-factorial foliated minimal models which are isomorphic in codimension one.
\begin{example}
    In the notation of example \ref{exflip}, consider the foliated triple $(X(\D_1),\mcF_1,\Sigma_1)$, where $\Sigma_1=D_{v_2}+D_{v_4}$. Since $D_{v_2}$ and $D_{v_4}$ are foliation non-invariant parts of the toric boundary, this is a lc foliated triple, with $K_{\mathcal{F}_1}+D_{v_2}+D_{v_4}\sim 0$. Similarly, the strict transformed triple $(X(\D_2),\mcF_2,\Sigma_2)$ is also a lc foliated triple with $K_{\mcF_2}+\Sigma_2\sim 0$ as $\Sigma_2$ is the non-invariant part of the torus boundary. Hence $(X(\D_1),\mcF_1,\Sigma_1)$ and $(X(\D_2),\mcF_2,\Sigma_2)$ are two non-isomorphic minimal models which are isomorphic in codimension one.
\end{example}

\begin{theorem}\label{flop}
    Let $(X, \mcF, \Delta)/U$ be a corank one lc foliated triple on a $\mbQ$-factorial normal projective threefold with $(X, B)$ klt for some $B \leq \Delta$ and $\alpha_i: (X, \mcF, \Delta) \dashrightarrow (X_i, \mcF_i, \Delta_i)$, $i=1,2$ be two minimal models obtained as outcomes of some $(K_\mcF+\Delta)$-MMPs $\alpha_i:X \dashrightarrow X_i$. Then the induced birational map $\alpha: X_1 \dashrightarrow X_2$ can be realized as a sequence of $(K_{\mcF_1}+\Delta_1)$-flops.
\end{theorem}
\begin{proof} We  first show the following:\\

   \emph{Step 1:} There exists no log canonical center $W \subset X_1$ of $(X_1, \mcF_1, \Delta_1)$ contained in the exceptional locus $\Ex \alpha$. \\

   If not, we can choose a common birational model $\Tilde{X}$ of $X$, $X_1$ and $X_2$ such that there exists a prime divisor $E \subset \Tilde{X}$ with center $W$ on $X_1$ and $a(E, X_1, \mcF_1, \Delta_1) = -\epsilon(E)$. Let $p: \Tilde{X} \to X_1$, $q: \Tilde{X} \to X$ and $\Tilde{X} \to X_2$ denote the induced morphisms. Then $E_X:=c_X(E)$ is a lc center of $(X, \mcF, \Delta)$ (since by application of negativity lemma, an MMP can't create new lc centers) and so is $c_{X_2}(E)$ (since $p^*(K_{\mcF_1}+\Delta_1)=  q^*(K_{\mcF_2}+\Delta_2)$ as can be checked by two applications of the negativity lemma). Now note that either $\alpha_1$ or $\alpha_2$ can't be an isomorphism at the generic point $ \eta_W$ of $W$ (because otherwise, $\alpha$ would be an isomorphism at $\eta_W$). But then, we can apply \cite[Lemma 3.38]{KM98} to conclude that either $W$ or $c_{X_2}(E)$ can't be an lc center. Either possibility leads to a contradiction.\\

   \emph{Step 2:} $\alpha$ is an isomorphism in codimension one.\\

   For $i=1,2$, a prime divisor $D\subset X$ is contracted by $\alpha_i$ iff  $a(D,X, \mcF, \D) <a(D,X_i, \mcF_i, \D_i)$. Since $a(D,X_1, \mcF_1, \D_1)=a(D,X_2, \mcF_2, \D_2)$ by two applications of negativity lemma, $D$ is $\alpha_1$-exceptional iff it is $\alpha_2$-exceptional. Thus $\alpha$ is an isomorphism in codimension one.\\

   \emph{Step 3:} We can run a carefully chosen MMP which is $(K_{\mcF_1}+\Delta_1)$-trivial to go from $X_1$ to $X_2$.\\

   In this step, we let $p:\Tilde{X} \to X_1$, $q:\Tilde{X}\to X_2$ denote the normalization of the closure of the graph of $\alpha:X_1\dashrightarrow X_2$. Let $A_2$ be an ample divisor on $X_2$, $A$ and $A_1$ denote its strict transforms on $X$ and $X_1$ respectively. We can consider $(X_2,\mcF_2, \Delta_2, \A)$ and $(X_1,\mcF_1, \Delta_1, \A)$ as gfqs (here the trace of $A_2$ on various birational models of $X_2$ is given by  Cartier closure) and note that the former is an lc gfq. Also note that if $\pi: \hat{X} \to X_2$ is a higher model, then by choosing $A_2$ general in its linear system, we may assume that $\pi^*A_2=\pi_*^{-1}A_2$). We now proceed to show using Step $1$ that after possibly rescaling $\A$, $(X_1, \mcF_1, \Delta_1, \A)$ is lc. Note that $p^*(A_1)=q^*(A_2)+E$ for some $E \geq 0$ that is exceptional over $X_1$ and $X_2$, in particular, Supp $p(E)\subset \Ex \alpha$. Then $p^*(K_{\mcF_1}+\Delta_1+A_1)=q^*(K_{\mcF_2}+\Delta_2+A_2)+E$. Thus if $\epsilon>0$ is small enough, by Step $1$, $(X_1,\mcF_1, \Delta_1,\epsilon \A)$ is lc. Replacing $\epsilon$ with something still smaller, we can ensure that $\alpha_i$ are both negative with respect to $K_\mcF+\Delta+ \epsilon A$. From now on, we will replace $\epsilon A_i$ with $A_i$ for $i=1,2$. \\

   We may run a $(K_{\mcF_1}+\Delta_1+A_1)$-MMP over $U$, say $(X_1, \mcF_1, \Delta_1,\A) \dashrightarrow \cdots \dashrightarrow (X', \mcF', \Delta',\A)$. Note that we have the termination of any such MMP by corollary \ref{glcterm}. Since any two minimal models of $(X, \mcF, \Delta, \A)$ are crepant, it follows that $K_{\mcF'}+\Delta'+\A'$ is semiample and the induced birational map $X' \dashrightarrow X_2$ is its semiample fibration. If this birational morphism is not an isomorphism, its exceptional locus is divisor (because of $\mbQ$-factoriality), but this is impossible since $X_1$ and $X_2$ 
   are isomorphic in codimension one and $(X', \mcF', \Delta',\A)$, being a minimal model of $(X_1, \mcF_1, \Delta_1,\A)$, can't contain any extra divisors. In other words, any $(K_{\mcF_1}+\Delta_1+\A_1)$-MMP gets us from $X_1$ to $X_2$. Thus from now on, our concern will be running this MMP in a $(K_{\mcF_1}+\Delta_1)$-trivial manner. \\

   \emph{Claim:} There exists $t \in (0, 1]$ such that the $(K_{\mcF_1}+\Delta_1+t\A_1)$-MMP is $(K_{\mcF_1}+\Delta_1)$-trivial. \\

   First, it is easy to check that the arguments of \cite[Proposition 3.2(3)]{Bir09} (see also \cite[Theorem 1.12]{MMPlcintegrable}) work in our setting as well. Thus, $\mathcal{N}:=\{\Theta \geq 0| (X_1,\mcF_1, \Theta)$ is lc and $(K_{\mcF_1}+\Theta) \cdot R \geq 0$ for all extremal rays $R \subset \overline{NE}(X/U)\}$ is a rational polytope containing $\Delta_1$ and we can write $K_{\mcF_1}+\Delta_1 = \sum_{i=1}^n a_i L_i$, where $a_i > 0$ for all $i$, $\sum_{i=1}^na_i=1$ and the $L_i$ are nef $\mbQ$-Cartier $\mbQ$-divisors which are of the form $K_{\mcF_1}+\Theta_i$ where $(X_1,\mcF_1, \Theta_i)$ is lc for all $i$.\\
   
   Choose $k \in \mathbb{N}$ such that $kL_i$ is Cartier for all $i$. Let $a:= \rm{min}\{a_1, \cdots, a_n\}$ and $e:= \frac{\frac{a}{k}}{6+\frac{a}{k}}$.
   As observed above, we may assume $K_{\mcF_1}+\Delta_1+e \A_1$ is not nef. Let $R \subset \overline{NE}(X_1/U)$ be a negative extremal ray with respect to it. Then $R$ is also $(K_{\mcF_1}+\Delta_1+\A_1)$-negative. Thus by Theorem \ref{glccone}, $R = \mbR_+[C]$ for some rational curve $C\subset X_1$ tangent to $\mcF_1$ with $0>(K_{\mcF_1}+\Delta_1+\A_1)\cdot C \geq -6$. Suppose that $(K_{\mcF_1}+\Delta_1) \cdot C >0$ and that $L_i \cdot C >0 $ for $i=1, \cdots ,m$ and $L_i \cdot C =0$ for $i=m+1, \cdots, n$. Since if $L_i \cdot C>0$, $L_i \cdot C \geq \frac{1}{k}$, this gives $(K_{\mcF_1}+\Delta_1) \cdot C \geq \frac{a_1+ \cdots a_m}{k} \geq \frac{a}{k}$. With this, we have
   \begin{center}
       $(K_{\mcF_1}+\Delta_1+e\A_1) \cdot C = e(K_{\mcF_1}+\Delta_1+\A_1) \cdot C+$ \newline $ (1-e)(K_{\mcF_1}+\Delta_1) \cdot C \geq -6e +(1-e)\frac{a}{k}= 0$
   \end{center}
   which is a contradiction. We conclude $(K_{\mcF_1}+\Delta_1) \cdot C =0$. Taking $t:=e$ gives our claim. Indeed, since the first step of the MMP is $(K_{\mcF_1}+\Delta_1)$-trivial (hence also $L_i$-trivial for all $i$) the Cartier index of $L_i$ stays the same by Theorem \ref{CT}. In particular, the same arguments also apply to subsequent steps. This finishes the proof of the theorem.
\end{proof}

\section{Log geography of minimal models}

In this section, we show that the number of minimal models of a boundary polarized lc gfq is finite. This can be seen as a complement to Theorem \ref{flop} above. More generally, we have the following result in the spirit of \cite{BCHM}. In case the foliation has non-dicritical singularities, this was obtained in \cite[Theorem 3.2]{Rok}. We first define the objects that will be required to state the theorem. They are analogous to those in \cite{BCHM}. 
\begin{definition}\cite[Definition $1.1.4$]{BCHM}\label{polytope}
    Let $X$ be a normal projective variety equipped with a projective morphism $\pi: X \to U$.  Let $\M= \sum_{i=1}^ka_iM_i$ be a NQC b-divisor on $X$ where $a_i$ are positive real numbers and $\M_i$ are $\mbQ$ b-nef divisors over $U$. Let $Span_{\mbR}(\mathcal{M}):=\lbrace\sum_{i=1}^{k}a_i\M_i|a_i\in \mbR\rbrace$. $V$ be a finite-dimensional affine subspace of $\text{WDiv}_{\mbR}(X)\times Span_{\mbR}(\mathcal{M})$. Fix a $\pi$-ample $\mbR$-divisor $A$ on $X$. Let $\mathcal{F}$ be a foliation on $X$. Then we define
    \begin{enumerate}
        \item $V_{A}=\lbrace (B,\M+A)| (B,\M)\in V\rbrace$
        \item $\mathcal{L}_{\pi,A}(V)=\lbrace (B,\M+A)\in V_{A}|(X,\mcF, B, \M+A) \text{ is a lc gfq over $U$ and } B\geqslant 0\rbrace$ 
        \item $\mathcal{E}_{\pi,A}(V)=\lbrace (B,\M+A)\in \mathcal{L}_{\pi,A}(V)|K_{\mathcal{F}}+\Delta+\M_X+A\text{ is pseudo-effective over $U$ }\rbrace$
        
        \item Given a birational contraction $\phi:X\dashrightarrow Y$ over $U$, define \[ 
        \mathcal{L}_{\pi,\phi,A}(V)=\lbrace (B,\M+A)\in \mathcal{E}_{\pi,A}(V)| \phi \text{ is a minimal model for $K_\mcF+\Delta+\M$ over $U$}\rbrace\]
        \item And finally given a rational map $\psi:X\dashrightarrow Z$ over $U$, we define
        \[
        \mathcal{A}_{\pi,\psi,A}(V)=\lbrace (B,\M+A)\in \mathcal{E}_{\pi,A}(V)| \psi \text{ is the ample model for $K_{\mathcal{F}}+\Delta +\M$ over $U$}\rbrace
        \]
        
    \end{enumerate}

 \end{definition} 

\begin{theorem}\label{LGLCM}
    Let $X$ be a normal projective threefold equipped with a corank one foliation $\mcF$ and a NQC nef over $U$ b-divisor $\M$. With notation as in Definition \ref{polytope},let $V$ be a finite-dimensional affine subspace of $\text{WDiv}_{\mbR}(X)\times Span_{\mbR}(\mathcal{M})$ defined over the rationals. Let $A$ be a general ample $\mbR$-divisor.
    \begin{enumerate}
        \item There are finitely many birational contractions $\phi_i:X\dashrightarrow Y_i$ over $U$,
        $1\leqslant i\leqslant p$ such that \[
        \mathcal{E}_{\pi,A}(V)=\cup_{i=1}^{p} \mathcal{L}_i
        \]
        where each $\mathcal{L}_i=\mathcal{L}_{\pi,\phi_i,A}(V)$ is a rational polytope. Moreover, if $\phi:X \dashrightarrow Y$ is a minimal model of $(X,\mathcal{F},\Delta,\M)$ over $U$, for some $(\Delta,\M)\in \mathcal{E}_{\pi,A}(V)$, then $\phi=\phi_i$ for some $1\leqslant i\leqslant p$
        \item There are finitely many rational maps $\psi_j:X\dashrightarrow Z_j$, $1\leqslant j\leqslant q$  which partition $\mathcal{E}_{\pi,A}(V)$ into subsets $\mathcal{A}_j=\mathcal{A}_{\pi,\psi_j,A}(V)$
        \item For every $1\leqslant i\leqslant p$ there is a $1\leqslant j\leqslant q$ and a morphism $f_{i,j}:Y_i\rightarrow Z_j$ such that $\mathcal{L}_i\subset \overline{\mathcal{A}}_j$.
    \end{enumerate}
    In particular $\mathcal{E}_{\pi,A}(V)$ and each $\overline{\mathcal{A}}_j$ are rational polytopes.
\end{theorem}
\begin{proof}
    First of all, note that $\mathcal{E}_{\pi,A}(V)$ is a rational polytope and for $(\Delta ,\M)\in \mathcal{E}_{A,M}(V)$ we know the existence of a minimal model.
    We will use induction on dimension of $\mathcal{E}_{\pi,A}(V)$. For the base case, dimension of $\mathcal{E}_{\pi,A}(V)$ is $0$ i.e. it is a point $(\Delta ,\M)$. Then we already know the the existence of minimal model and by basepoint free theorem we know the existence of log canonical model.\\
    
    So we assume dimension of $\mathcal{E}_{\pi,A}(V)$ is strictly positive. First asssume there is a $(\Delta_0 ,\M_0)\in \mathcal{E}_{\pi,A}(V)$ such that $K_{\mathcal{F}}+\Delta_0 +\M\sim_{\mbR}0$. Pick $(\Theta ,\M)\in\mathcal{E}_{\pi,A}(V)$ different from $(\Delta_0 ,\M_0)$, then there is a $(\Delta ,\M)$ in the boundary of $\mathcal{E}_{\pi, A}(V)$ such that $(\Theta ,\M)=\lambda (\Delta ,\M)+(1-\lambda)(\Delta_0 ,\M_0)$ for some $0<\lambda\leqslant 1$. Hence $K_{\mathcal{F}}+\Theta +\M_X\sim_{\mbR}\lambda(K_{\mathcal{F}}+\Delta +\M) $ and they have the same minimal model. On the other hand, we know that there are only finitely many minimal models for all $(\Delta ,\M)$ contained in the boundary of $\mathcal{E}_{\pi,A}(V)$ by the induction hypothesis. So we are done in this case.\\
    
    Now we come to the general case. Since $\mathcal{E}_{\pi,A}(V)$ is compact, it is sufficient to prove the finiteness of minimal models in a neighbourhood of a fixed $(\Delta_0, \M_0) \in \mathcal{E}_{\pi, A}(V)$. Let $\phi:X\dashrightarrow Y$ be a minimal model of $K_{\mathcal{F}}+\Delta_0+\M_0$ with $\pi': Y \to U$ the induced contraction. We choose a neighbourhood $\mathcal{C}_0$ of $(\Delta_0, \M_0)$ such that $\phi$ is also $(K_{\mathcal{F}}+\Delta+M)$-nonpositive for all $(\Delta, \M)\in \mathcal{C}_0$. Then $\phi_*(\mathcal{C}_0)$ is contained in $\mathcal{E}_{\pi',\phi_*(A)}(\phi_*(V))$. As in the proof of Corollary \ref{gmm} we can find ample divisor $A'$ in $Y$ such that $K_{\mathcal{F}_Y}+\Delta_Y+\M\sim_{\mbR}K_{\mathcal{F}_Y}+\Delta'_Y+\M'$ for all $(\Delta_Y,\M)\in\phi_*(\mathcal{C}_0) $, where $(\Delta'_Y,\M')\in \mathcal{E}_{A',\M'}(W)$, and $W\subset WDiv_{\mbR}(Y)$ is finite dimensional. As $\phi$ is $(K_{\mathcal{F}}+\Delta+\M)$-nonpositive for all $(\Delta,\M)\in \mathcal{C}_0$ we know that any minimal model for $K_{\mathcal{F}}+\Delta+\M$ is also a minimal model of $K_{\mathcal{F}_Y}+\Delta'_Y+\M'$. Hence replacing $X$ by $Y$ and $\mathcal{C}_0$ by the corresponding polytope in $Y$, we can assume $K_{\mathcal{F}}+\Delta_0+\M_0$ is nef, and Theorem \ref{bpft} implies it is semiample. Consider the corresponding ample model $\psi:X\rightarrow Z$, then $K_{\mathcal{F}}+\Delta_0+\M_0\sim_{\mbR,Z} 0$. By the arguments of the above paragraph, we know that for all $(\Delta,\M)\in \mathcal{C}_0$ there exist finitely many minimal models $\phi_i:X\dashrightarrow Y_i$ of $K_\mcF+\Delta+\M$ over $Z$. We claim that we can find a neighbourhood $P_0$ of $(\Delta_0,\M_0)$ such that for all $(\Delta,\M)\in P_0$, if $\phi_i:X\dashrightarrow Y_i$ is a minimal model for $K_{\mathcal{F}}+\Delta+\M$ over $Z$, then it is also a minimal model globally. Indeed, if $(K_{\mcF_{Y_i}}+\Delta_{Y_i}+\M) \cdot R<0$ for some extremal ray $R$ spanned by a curve which is horizontal over $Z$, then $(K_{\mcF_{Y_i}}+\Delta_{0,Y_i}+\M_0) \cdot R >0$, since $K_{\mcF_{Y_i}}+\Delta_{0,Y_i}+\M_0$ is the pullback of an ample divisor on $Z$. Thus, if $(\Delta,\M)$ is sufficiently close to $(\Delta_0,\M_0)$, $(K_{\mcF_{Y_i}}+\Delta_{Y_i}+\M_X) \cdot R >0$ as well, a contradiction. Hence there are only finitely many minimal models for all $(\Delta,\M)\in P_0$, and by the compactness of $\mathcal{E}_{\pi,A}(V)$ we are done. These finitely many minimal models give rise to the finite partition of $\mathcal{E}_{\pi,A}(V)$, which proves $(1)$.\\
    
    For $(2)$, note that by Theorem \ref{bpft} we have the existence of a log canonical model. And since the log canonical model (or ample model) is unique, hence from the finiteness of minimal models, it follows that there can be only finitely many log canonical models $\psi_j:X\dashrightarrow Z_j$ for all $(\Delta,\M)\in \mathcal{E}_{\pi,A}(V)$.\\
    
    Finally for $(3)$, if $\mathcal{L}_i\subset \overline{\mathcal{A}_j}$ then the map $f_{i,j}:Y_i\rightarrow Z_j$ is the corresponding semiample fibration for some $(\Delta,\M)$ inside the boundary of $\overline{\mathcal{A}_j}$ and $\mathcal{L}_i$.
    
\end{proof}
\section{Appendix: Towards a general contraction theorem}
In this paper, we have developed the minimal model program for lc foliated triples on threefolds of klt type assuming that the klt boundary is at most the foliated boundary. In case we make no assumptions on the klt boundary, we show below that negative extremal rays can still be contracted at the level of algebraic spaces.\footnote{More recent developments indicate the arguments used in \cite[Theorem 2.1.3]{M7'} can be used to show the contraction exists at the level of varieties.}

\begin{theorem}\label{nqfCT} Let $(X,\mcF, \Delta)/U$ be a projective rank two foliated lc triple with a projective morphism $\pi: X \to U$, where $X$ is a threefold of klt type. Let $R \subset\overline{NE}(X/U)$ be a $(K_\mcF+ \Delta)$-negative exposed extremal ray. Then there exists a projective contraction $\phi: X \to \mathcal{Y}$ to an algebraic space $\mathcal{Y}$ such that $\phi(C)= pt$ iff $[C] \in R$. 
\begin{proof} We deal with the case $U=\rm{pt}$. As in Theorem \ref{CT}, similar arguments work in the general case also by replacing the absolute MMPs with relative ones.
First suppose loc $R \neq X$.
Let $\hat{\pi}:\hat{X} \to X$ be a small $\mbQ$-factorialization of $X$ and $\pi: \overline{X}\xrightarrow{\overline{\pi}}\hat{X} \xrightarrow{\hat{\pi}}X$ a $\mbQ$-factorial F-dlt modification of $(X, \mcF, \Delta)$. Let $H_R=K_\mcF+\Delta+A$ be a supporting Cartier divisor of $R$, where $A$ is ample, $\overline{H_R}:=\pi^*H_R$, $\overline{A}:=\pi^*A$. Let $\overline{\mcF}$ be the induced foliation on $\overline{X}$ and $\overline{\Delta}$ be defined by $K_{\overline{\mcF}}+\overline{\Delta}=\pi^*(K_\mcF +\Delta)$.\\

     Suppose loc $R=D$ is a divisor (possibly with several components) and $\overline{D}:=\pi_*^{-1}D$. The idea is to run (some steps of) a carefully chosen $(K_{\overline{\mcF}}+\overline{\Delta})$-MMP to contract $\overline{D}$ to a union of curves, then run a usual MMP to contract $\Ex \pi$ and finally argue as in the proof of \cite[Lemma 8.12]{Spi} to get the desired contraction to an algebraic space.\\

     Let $\overline{X} \dashrightarrow \overline{X_1}\dashrightarrow \cdots\overline{X_i}\cdots$ be steps of a $(K_{\overline{\mcF}}+\overline{\Delta})$-MMP with scaling of $\overline{A}$ and $\lambda_i:= \rm{inf} \{t: K_{\overline{\mcF_i}}+\overline{\Delta_i}+t\overline{A_i}$ is nef$\}$. Then by termination of $K_{\overline{\mcF}}+\overline{\Delta}$ flips, there exists $i_0$ such that $\lambda_{i_0}<1$; choose $i_0$ to be the smallest such index, let $\overline{X}':=\overline{X}_{i_0}$ and $\phi: \overline{X}\dashrightarrow \overline{X}'$ denote the induced birational map. Then $\phi$ is a full $(K_{\overline{\mcF}}+\overline{\Delta}+\lambda_{i_0}\overline{A})$-MMP all the steps of which are $\overline{H_R}$-trivial. Note also that $K_{\overline{\mcF}'}+\overline{\Delta}'+t \overline{A}'$ is nef for all $\lambda_{i_0} \leq t \leq 1$. Since $\overline{D}\subset B_-(K_{\overline{\mcF}}+\overline{\Delta}+\lambda_{i_0}\overline{A})$, $\phi$ contracts all components of $\overline{D}$ by \cite[Lemma 2.1]{CS3}.\\

     Since $\hat{X}$ is $\mbQ$-factorial klt, if $\hat{\Delta}:= \hat{\pi}_*^{-1}\Delta$, then $(\hat{X}, \epsilon \hat{\Delta})$ is klt for $\epsilon>0$ small enough. Write $\overline{\pi}^*(K_{\hat{X}}+\epsilon\hat{\Delta})+E_0=K_{\overline{X}}+\epsilon\overline{\Delta}+F_0$, where $E_0, F_0 \geq 0$ are $\overline{\pi}$-exceptional divisors without common components. Let $B\geq 0$ be $\overline{\pi}$-exceptional such that $-B$ is $\overline{\pi}$-ample. We can choose $\delta>0$ such that if $E:= E_0+\delta B$, $F:=F_0+\delta B$, then $(\overline{X}, \epsilon \overline{\Delta}+F)$ is klt (note that $(\overline{X}, \epsilon \overline{\Delta})$ is klt by \cite[Lemma 3.16]{CS}). \\

     Let $\overline{K}:=\overline{H_R}-\lambda \overline{A} + s(K_{\overline{X}}+\epsilon \overline{\Delta} +F)$, where we fix $\lambda>0$ small with $\lambda_{i_0}<1-\lambda <1$ and $s>0$ small enough such that $\overline{K}$ is big. Let $\psi: \overline{X}' \dashrightarrow \overline{X}''$ be a $\overline{K}'$-MMP; note that for $s>0$ sufficiently small, \begin{enumerate}
         \item this MMP is $\overline{H_R}'-\lambda \overline{A}'$-trivial, hence also $\overline{H_R}$-trivial by \cite[Lemma 3.1]{CS3}.
         \item $\psi \circ \phi$ is $\overline{K}$-negative.
     \end{enumerate}
     Then $\overline{K}''$ is semiample by the classical basepoint free theorem for klt pairs. Now note that $\overline{K}= \overline{\pi}^*[\hat{H}_R -\lambda \hat{A}+ s(K_{\hat{X}}+\epsilon \hat{\Delta})]+sE$. Thus looking at the $\mbQ$-linear systems on both sides, we conclude that $E \subset B(K)$, where $B(K)$ is the stable base locus of $K$ and thus $\psi \circ \phi$ contracts $E$ to a bunch of points. \\

     We now claim that $B_+(\overline{H}_R'')$ is a finite union of curves. For this, since $\overline{H}_R''$ is nef and big, by Nakamaye's theorem, it suffices to show that $\overline{H}_R''^2 \cdot S >0$ for any surface $S \subset \overline{X}''$ (Note that our arguments here are similar to those in \cite[Lemma 8.20]{Spi}). Let $\nu: S^\nu \to \overline{X}''$ denote the normalization of $S$. For a divisor $D$ on $\overline{X}''$, we abuse notation and denote $\nu^*D$ by $D|_S$.\\

     Note the nummerical dimension $\nu(\overline{H}_R''|_S)$ can not equal $0$, since our MMPs have contracted $\overline{D}$ and $\Ex \pi$ in a $\overline{H_R}$-crepant way, so in particular, $\overline{H}_R''$ can not be numerically trivial on a divisor. \\

     Now let $\nu(\overline{H}_R''|_S)=1$ and $\overline{S}$ denote the strict transform of $S$ on $\overline{X}$. Then ${\overline{H}_R''|_S}^2=0$ and $\overline{A}|_{\overline{S}}\cdot \overline{H}_R|_{\overline{S}}>0$. The $\overline{H}_R$ crepancy of $\psi \circ \phi$ then shows that $(K_{\overline{\mathcal{F}}''}+\overline{\Delta}'')|_S \cdot \overline{H}_R''|_S= - \overline{A}''|_S \cdot \overline{H}_R''|_S<0$. Possibly after rescaling, we can write $\overline{H}_R''=A''+D+S$, where $D \geq 0$ and shares no common components with $S$. Using this, we see that $\overline{H}_R''|_S \cdot S|_S= -\overline{H}_R''|_S \cdot (A''+D)|_S \leq - \overline{H}_R''|_S \cdot A''|_S <0$. We now divide into two cases depending on whether $S$ is invariant  or not.\\

     Suppose $\epsilon(S)=0$. Then there exists $\Theta \geq 0$ such that $(K_{\overline{\mcF}''}+\overline{\Delta}'')|_S \cdot \overline{H_R}''|_S=(K_S +\Theta)\cdot \overline{H}_R''|_S<0 $ by adjunction and the observation in the above paragraph. Thus by bend and break, through a general point of $S$, there passes a $\overline{H_R}''$-trivial rational curve, which is impossible since $\overline{H_R}''$ can not be numerically trivial along a moving family of curves. Now, suppose $\epsilon(S)=1$. We can choose $0 \leq t \leq 1$ such that $\overline{\Delta}''+tS=\Delta'+S$, where $\Delta' \geq 0$ and $S$ have no components in common and by adjunction, there exists $\Theta \geq 0$ such that letting $\overline{\mcF}''_S$ denote the restricted foliation on $S$, we have $(K_{{\overline{\mcF}''}}+\overline{\Delta}''+tS)|_S= K_{\overline{\mcF}''_S}+\Theta$. Finally, by the observations in the above paragraph, we have $(K_{\overline{\mcF}''_S}+\Theta) \cdot \overline{H_R}''|_S <0$. Then by another application of \cite[Corollary 2.28]{Spi}, through a general point of $S$, there passes a $\overline{H_R}''$-trivial rational curve tangent to $\overline{\mcF}''$, again a contradiction. Thus $\nu(\overline{H_R}''|_S)=2$ is the only possibility, thereby proving the claim.\\

     Since $B_+(\overline{H_R}'')$ is a union of curves, we can argue as in the proof of \cite[Lemma 8.21]{Spi} (by observing that this union of curves has negative definite interseection matrix and then using Artin's theorem to contract it) to get the desired morphism $\phi: X \to \mathcal{Y}$ in case loc $R$ is a divisor.\\

     If loc $R$ is a curve, then we can directly apply the arguments of the above paragraph to a small $\mbQ$-factorialization of $X$ to get the desired morphism.\\

     Now suppose loc $R=X$. Let $\pi: \overline{X} \to X$ be a small $\mbQ$-factorialization. Let $\overline{R} \subset \overline{NE}(\overline{X})$ be the extremal ray with $\pi_*\overline{R}=R$. Then $(K_{\overline{F}}+\overline{\Delta})\cdot \overline{R}<0$. Let $\phi:= c_{\overline{R}}:\overline{X} \to S$ be the associated contraction given by Theorem \ref{CT}; note that $\dim S <3$. Let $H_R:= K_\mcF +\Delta +A$ be a supporting Cartier divisor of $R$ (where $A$ is an ample $\mbQ$-divisor. Then $\overline{H}_R:=\pi^*H_R \sim _{\mbR,S}0$; thus $\overline{H}_R \sim_\mbR \phi^*L_S$ for some nef divisor $L_S$. We argue separately depending on whether $L_S$ is big or not.\\
     
     Suppose $L_S$ is big and $\dim S=2$. Let $C^i\subset\overline{X}$ be the curves extracted by $\pi$ and $C^i_S:=\phi(C^i)$. Since $\cup_i C^i_S$ is the null locus of the nef and big divisor $L_S$, it has a negative definite intersection matrix. Thus by Artin's theorem, there exists a birational morphism $\pi: S \to \mathcal{S}'$ to an algebraic space contracting $\cup_i C^i_S$. This induces a morphism $c_R : X \to \mathcal{S}'$.\\

     Suppose $L_S$ is not big. By Theorem \cite[Theorem 2.3.2]{CHLX}, we can write $L_S=K_\mcG+\Delta_S+\N_S$, where $\mcG$ is the descended foliation on $S$ and $(S,\mcG, \Delta_S, \N)$ is an lc gfq with semiample moduli part. In particular, this allows us to contract the null locus of $L_S$ by \cite[Theorem 2.8]{SS23}, thus finishing the proof.

\end{proof}
    
\end{theorem}
\textbf{Acknowledgements}: The authors would like to thank Professor Paolo Cascini, Professor Calum Spicer and Professor Roberto Svaldi for many useful discussions and the referees for their comments which helped improve the quality of the paper.\\

\textbf{Funding:} PC is a member of GNSAGA group of the Istituto Nazionale di Alta Matematica "Francesco Severi" (INDAM) and gratefully acknowledges financial support from INDAM during the preparation of this work. Part of the work in this version of the paper was done at the Yau Mathematical Sciences Center and PC thanks the Center for providing financial support and excellent working environment. RM was supported by the Engineering and Physical Sciences Research Council [EP/S021590/1]. The EPSRC Centre for Doctoral Training in Geometry and Number Theory (The London School of Geometry and Number Theory), University College London.\\




\bibliographystyle{hep}
\bibliography{references}

\end{document}